\documentclass[12pt]{amsart}
\usepackage{fullpage}
\usepackage{amsfonts}      
\usepackage{amsmath}
\usepackage{amsthm}
\usepackage{amssymb} 
\usepackage{cite}
\usepackage{setspace}
\usepackage{caption}
\usepackage{float}
\usepackage{graphicx}
\usepackage{marvosym}
\usepackage{pst-node}
\usepackage{stmaryrd}
\usepackage{tikz}
\usepackage{tikz-cd}
\usepackage{multicol}
\usepackage[utf8]{inputenc}
\usepackage[T1]{fontenc}
\usepackage{upgreek}
\usepackage{xcolor}
\usepackage[colorlinks=true]{hyperref}

\usetikzlibrary{shapes,arrows}
\usetikzlibrary{decorations.markings}
\usetikzlibrary{decorations.pathreplacing}
\usetikzlibrary{cd}

\newcommand{\Q}{\mathbb{Q}}
\newcommand{\R}{\mathbb{R}}
\newcommand{\Z}{\mathbb{Z}}
\newcommand{\ZZ}{\mathbb{Z}}

\newcommand{\E}{\mathbf{E}}
\renewcommand{\P}{\mathbf{P}}
\newcommand{\PP}{\mathbb{P}}

\newcommand{\la}{\langle}
\newcommand{\ra}{\rangle}
\newcommand{\td}{\widetilde}
\newcommand{\avg}{\operatorname{avg}}

\DeclareMathOperator{\maxdeg}{maxdeg}

\DeclareMathOperator{\Tor}{Tor}

\DeclareMathOperator{\Var}{Var}

\theoremstyle{plain}
	\newtheorem{theorem}{Theorem}[section]
	\newtheorem{lemma}[theorem]{Lemma}
    \newtheorem{corollary}[theorem]{Corollary}
    \newtheorem{proposition}[theorem]{Proposition}
    \newtheorem{conjecture}[theorem]{Conjecture}
    \newtheorem{fact}[theorem]{Fact}
    \newtheorem{question}[theorem]{Question}
\theoremstyle{definition}
    \newtheorem{defn}[theorem]{Definition}
    
    \newtheorem{notation}[theorem]{Notation}
\theoremstyle{remark}
	\newtheorem{remark}[theorem]{Remark}

\newcommand{\ds}{\displaystyle}
\usepackage[top=1in, bottom=1in, left=1in, right=1in]{geometry}

\title{Characteristic dependence of syzygies of random monomial ideals}
\author{Caitlyn Booms, Daniel Erman, and Jay Yang}
\thanks{The authors were supported by NSF GRFP grant DGE-1747503 and by NSF grants DMS-1502553, DMS-1902123, and DMS-1745638. Support was also provided by the Graduate School and OVCRGE at UW-Madison with funding from the Wisconsin Alumni Research Foundation.}

\begin{document}

\maketitle

\begin{abstract}
When do syzygies depend on the characteristic of the field? Even for well-studied families of examples, very little is known. For a family of random monomial ideals, namely the Stanley--Reisner ideals of random flag complexes, we prove that the Betti numbers asymptotically almost always depend on the characteristic. Using this result, we also develop a heuristic for characteristic dependence of asymptotic syzygies of algebraic varieties.
\end{abstract}

\section{Introduction}
The minimal free resolution of an ideal can depend on the characteristic of the ground field.  Known examples include certain monomial ideals~\cite{dalili-kummuni, katzman}, Veronese embeddings of $\PP^r$~\cite{anderson, jonsson-experimental}, and determinantal ideals~\cite{hashimoto}.  
This paper is motivated by a desire to understand if dependence on the characteristic is a common or rare phenomenon. To make such a question precise, we can restrict to specific families, such as:
\begin{question}\label{q:ver}
For which $d\geq 1$ does the minimal free resolution of the $d$-uple embedding of $\PP^r$ depend on the characteristic?  Does it happen for all $d\gg 0$?  Or does it happen rarely?
\end{question}
\begin{question}\label{q:rm}
Let $\Delta\sim \Delta(n,p)$ be a random flag complex (see \S\ref{subsec:random flag}).  As $n\to \infty$, what is the probability that the minimal free resolution of the Stanley--Reisner ideal of $\Delta$ depends on the characteristic?
\end{question}

We do not offer new results on Question~\ref{q:ver}, though we discuss in \S\ref{subsec:asymptotic} how questions like this motivated our work. Our main result is Theorem~\ref{thm:m torsion}, which answers Question~\ref{q:rm} and shows that in this context, dependence on the characteristic is quite common.

To analyze dependence on characteristic, we will say that the Betti table of the Stanley--Reisner ideal of $\Delta$ \textbf{has $\ell$-torsion} if this Betti table is different when defined over a field of characteristic $\ell$ than it is over $\mathbb Q$.
See \S\ref{sec:background} for further details on notation.  We prove:

\begin{theorem}\label{thm:m torsion}
Let $\Delta\sim \Delta(n,p)$ be a random flag complex with $n^{-1/6} \ll p \leq 1-\epsilon$ for $\epsilon>0$. 
\begin{enumerate}
  \item  With high probability as $n\to \infty$, the Betti table of the Stanley--Reisner ideal of $\Delta$ depends on the characteristic.
    \item  More specifically, if we fix any $m\geq 2$, then with high probability as $n\to \infty$, the Betti table of the Stanley--Reisner ideal of $\Delta$ has $\ell$-torsion for every prime $\ell$ dividing $m$.
\end{enumerate}
\end{theorem}

The proof of Theorem~\ref{thm:m torsion} (2), which implies part (1), proceeds as follows.  By Hochster's formula~\cite[Theorem~5.5.1]{bruns-herzog}, it suffices to show that some induced subcomplex of $\Delta$ has $m$-torsion in its homology.  For each $m$, we modify Newman's construction~\cite[\S3]{newman} to build a flag complex $X_m$ with a small number of vertices and with  $m$-torsion in $H_1(X_m)$.  We then apply a variant of Bollob\'{a}s's theorem on subgraphs of a random graph~\cite[Theorem~8]{Bollobas-subgraph} to prove that $X_m$ appears as an induced subcomplex of $\Delta$ with high probability as $n\to \infty$, yielding Theorem~\ref{thm:m torsion}.

The most common example of characteristic dependence is Reisner's example, coming from a triangulation of $\mathbb R\mathbb P^2$~\cite[\S5.3]{bruns-herzog}.  Other previous research on characteristic {\em independence} of monomial ideals includes ~\cite{terai-hibi,katzman,hibi-kimura-murai} for edge ideals and ~\cite[Theorem 5.1]{dalili-kummuni} for monomial ideals with component-wise linear resolutions.

Theorem~\ref{thm:m torsion} also fits into an emerging literature on random monomial ideals.  
This began with~\cite{random-monomial}, which outlined an array of frameworks for random monomial ideals, including models related to random simplicial complexes such as~\cite{costa-farber, kahle}.
The average Betti table of a random monomial ideal is analyzed in ~\cite{average-random}, while ~\cite{stilverstein-wilburne-yang} examines threshold phenomena in random models from~\cite{random-monomial}.  
Banerjee and Yogeshwaran study homological properties of the edge ideals of Erd\H{o}s--R\'{e}nyi random graphs in~\cite{banerjee}.  
There is also~\cite{erman-yang}, which uses random monomial methods to demonstrate some asymptotic syzygy phenomena from~\cite{ein-lazarsfeld-asymptotic,ein-erman-lazarsfeld-random}.  And finally, Theorem \ref{thm:m torsion} is thematically connected with ~\cite{torsion-burst}, which analyzes torsion homology in random simplicial complexes (whereas Theorem~\ref{thm:m torsion} analyzes the simpler question of finding $m$-torsion in the homology of {\em some} induced subcomplex of $\Delta(n,p)$).

\subsection{Asymptotic syzygies and heuristics}\label{subsec:asymptotic}
One of our main motivations for studying Question~\ref{q:rm} is a belief that this will provide heuristic insights into more geometric questions like Question~\ref{q:ver}.  We now explain this connection in more detail.

The study of asymptotic syzygies, as introduced by Ein and Lazarsfeld in~\cite{ein-lazarsfeld-asymptotic}, examines the overarching behavior of syzygies of algebraic varieties under increasingly ample embeddings.  Specifically, Ein and Lazarsfeld fixed a smooth variety $X$ with a very ample line bundle $A$ and considered the syzygies of $X$ embedded by $dA$ for $d\gg 0$.  They proved an asymptotic nonvanishing result which showed that the limiting behavior essentially only depended on $\dim X$.
Other researchers then found comparable limiting behavior for other families from geometry~\cite{zhou-integral,ein-erman-lazarsfeld-quick} and combinatorics~\cite{cjw,erman-yang}.  In a similar vein, ~\cite{ein-erman-lazarsfeld-random} conjectured that the syzygies of smooth varieties should asymptotically converge to a normal distribution, in an appropriate sense; that conjecture was verified for the combinatorial families in~\cite{erman-yang}.



In short, work on asymptotic syzygies suggests that the overarching behavior will be similar across many geometric and combinatorial examples.  This is the context in which Questions~\ref{q:ver} and \ref{q:rm} are connected.  Whereas Ein and Lazarsfeld identified behavior in geometric settings which carried over to combinatorial settings, we look in the opposite direction: could a combinatorial result shed light on asymptotic syzygies in geometric examples?\footnote{A similar idea appears in~\cite{ein-erman-lazarsfeld-random}, where a random model based on Boij-S\"oderberg theory is used to generate quantitative conjectures about the entries of Betti tables.} 

The study of $\ell$-torsion is ripe for such a heuristic due to the lack of results and the difficulty of computing the Betti numbers of higher dimensional varieties.  
For instance, for Veronese embeddings of $\PP^r$, the only results on $\ell$-torsion are for the $2$-uple embedding (exploiting the combinatorial description of ~\cite{reiner-roberts}): Andersen's thesis~\cite{anderson} shows that the Betti table of the $2$-uple embedding of $\PP^r$ has $5$-torsion for any $r\geq 6$, and Jonsson generalized this to produce $\ell$-torsion for $\ell=3,5,7,11,$ and $13$ and for various $r$~\cite{jonsson-experimental}. 
See~\cite{bouc,hashimoto} for similar results.
But even for $d$-uple embeddings of $\PP^r$, there are no examples of torsion when $d>2$ and no conjectures for any fixed $r\geq 2$.

The random flag complex model used in this paper was previously studied in work of Erman and Yang~\cite[Theorem~1.3]{erman-yang}, and they showed that if $n^{-1/(r-1)} \ll p \ll n^{-1/r}$, then  the Betti table of the Stanley--Reisner ideal of $\Delta(n,p)$ exhibits some of the asymptotic behavior of $r$-dimensional varieties from ~\cite{ein-lazarsfeld-asymptotic}.  We view Theorem~\ref{thm:m torsion}, which holds for $n^{-1/(r-1)} \ll p \ll n^{-1/r}$ when $r\geq 7$, as providing a heuristic for $\ell$-torsion in the asymptotic syzygies of a smooth variety $X$ of $\dim X \geq 7$.  
For concreteness, in the case of $\PP^r$, we conjecture:

\begin{conjecture}\label{conj:dependence}
Let $r\geq 7$.  For any $d\gg 0$, the Betti table of $\mathbb P^r$ under the $d$-uple embedding depends on the characteristic.
\end{conjecture}

\begin{conjecture}\label{conj:bad primes}
Let $r\geq 7$. As $d \to \infty$, the number of primes $\ell$ such that the Betti table of $\PP^r$ under the $d$-uple embedding has $\ell$-torsion is unbounded.
\end{conjecture}
We will discuss some related conjectures and questions, in more detail, in \S\ref{sec:veronese}.

\smallskip

This paper is organized as follows.  In \S\ref{sec:background}, we review notation and background, including on Betti numbers, Hochster's formula, and random flag complexes.  \S\ref{sec:construction} contains our main construction in which we construct an explicit flag complex $X_m$ with $m$-torsion in homology; see Theorem~\ref{thm:Xm}.  In \S\ref{sec:subgraphs}, we apply a minor variant of Bollob\'{a}s's theorem on subgraphs of a random graph to show that, with high probability, $X_m$ appears as an induced subcomplex of $\Delta(n,p)$ for any $n^{-1/6}\ll p\leq 1-\epsilon$ where $\epsilon > 0$ and $m\geq 2$.  
In \S\ref{sec:2torsion}, we analyze the case of 2-torsion more closely, using the techniques from \S\ref{sec:subgraphs} to expand known results from ~\cite{costa-farber-horak}.  
In \S\ref{sec:Betti numbers}, we combine results from \S\ref{sec:subgraphs} with Hochster's formula to prove Theorem~\ref{thm:m torsion}.  
Finally, in \S\ref{sec:veronese}, we discuss questions about $\ell$-torsion in asymptotic syzygies.

\subsection*{Acknowledgments}
We thank Christine Berkesch, Kevin Kristensen, Rob Lazarsfeld, Andrew Newman, Victor Reiner, Gregory G.\ Smith, and Melanie Matchett Wood for helpful conversations.  We thank Claudiu Raicu and Steven Sam for thoughtful comments on an early draft.

\section{Background and Notation}\label{sec:background}

\subsection{Torsion in Betti tables}\label{subsec:22}
Throughout this paper we will analyze graded algebras, all of which have the following form: there is an ideal $J$ in a polynomial ring $T$ with coefficients in $\ZZ$, where $T/J$ is flat over $\ZZ$, and we are interested in specializations $(T/J)\otimes_{\ZZ} k$ to various fields $k$. Our results focus on graded algebras that arise as the Stanley--Reisner rings of simplicial complexes.  But there are many other potential examples, such as the coordinate rings of Veronese embeddings of projective space, Grassmanians, toric varieties, and so on.  The central questions of this paper are concerned with when the Betti numbers of such algebras depend on the characteristic of $k$.

Let $J$ be a monomial ideal in $T=\ZZ[x_1,\dots,x_n]$.  For a field $k$, the algebraic Betti numbers of $(T/J)\otimes_{\ZZ} k$ are given by
\[
\beta_{i,j}((T/J)\otimes_{\ZZ} k) := \dim_k \Tor^{T\otimes_{\ZZ} k}_i((T/J)\otimes_{\ZZ} k,k)_j.
\]
The collection of all of these Betti numbers is called the Betti table. Since field extensions are flat, Betti numbers are invariant under field extensions and will therefore be the same for any field of the same characteristic.   Semicontinuity implies that
$\beta_{i,j}((T/J)\otimes_{\ZZ} \mathbb Q)\leq \beta_{i,j}((T/J)\otimes_{\ZZ} \mathbb F_{\ell}).$
We say that the Betti table of $J$ {\bf has $\ell$-torsion} if this inequality is strict for some $i,j$, and we say that the Betti table of $J$ {\bf depends on the characteristic} if it has $\ell$-torsion for some prime $\ell$.

\begin{remark}\label{rmk:torsion}
Let $J$ be an ideal in $T=\ZZ[x_1,\dots,x_n]$ which is flat over $\ZZ$.  Let $S = T\otimes_{\ZZ} \mathbb F_{\ell}=\mathbb F_\ell[x_1,\dots,x_n]$ and $I=JS$.  By a standard argument, it follows that
\[
\dim_{\mathbb{F}_\ell} \Tor^{S}_i(S/I,\mathbb F_\ell)_j = \dim_{\mathbb{F}_\ell} (\Tor^T_i(T/J,\ZZ)_j\otimes_\ZZ \mathbb F_\ell) + \dim_{\mathbb{F}_\ell}(\Tor^{\ZZ}_1(\Tor^T_{i+1}(T/J,\ZZ)_j, \mathbb F_\ell)).
\]
In particular, the Betti table of $J$ has $\ell$-torsion if and only if one of the $\Tor_{i+1}^T(T/J,\ZZ)_j$ has $\ell$-torsion as an abelian group.
\end{remark}

\subsection{Graphs and simplicial complexes}
For a simplicial complex $X$, we write $V(X),$ $E(X),$ and $F(X)$ for the set of vertices, edges, and (2-dimensional) faces of $X$, respectively. We use $|*|$ to denote the number of elements in these sets. The degree of a vertex $v$, denoted $\deg(v)$, is the number of edges in $X$ containing $v$.  We write $\maxdeg(X)$ for the maximum degree of any vertex of $X$, and we write $\avg(X)$ for the average degree of a vertex in $X$.

For a pair of graphs $H,G$, we write $H\subset G$ if $H$ is a subgraph of $G$.  We write $H \overset{ind}{\subset} G$ if $H$ is an induced subgraph of $G$, that is, if the vertices of $H$ are a subset of the vertices of $G$ and the edges of $H$ are precisely the edges connecting those vertices within $G$ (see Figure~\ref{fig:induced graph}).  We use similar definitions and notations for a simplicial complex $\Delta'$ to be a subcomplex (or an induced subcomplex) of another complex $\Delta$.  If $\alpha \subset V(\Delta)$, then we let $\Delta|_{\alpha}$ denote the induced subcomplex of $\Delta$ on $\alpha$.

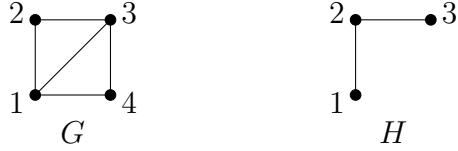
\begin{figure}
\begin{tikzpicture}[baseline=2ex]
    \draw[fill] (0,0) circle [radius=2pt];
    \draw[fill] (0,1) circle [radius=2pt];
    \draw[fill] (1,0) circle [radius=2pt];
    \draw[fill] (1,1) circle [radius=2pt];
    
    \draw (0,0)--(1,0)--(1,1)--(0,0)--(0,1)--(1,1);
    \draw   (.5,-.5) node {$G$};
    \draw   (-.25,-.1) node {$1$};
    \draw   (-.25,1.1) node {$2$};
    \draw   (1.25,1.1) node {$3$};
    \draw   (1.25,-.1) node {$4$};
\end{tikzpicture}
\hspace{2cm}
\begin{tikzpicture}[baseline=2ex]
    \draw[fill] (0,0) circle [radius=2pt];
    \draw[fill] (0,1) circle [radius=2pt];
    \draw[fill] (1,1) circle [radius=2pt];
    \draw (0,0)--(0,1)--(1,1);
    \draw   (.5,-.5) node {$H$};
    \draw   (-.25,-.1) node {$1$};
    \draw   (-.25,1.1) node {$2$};
    \draw   (1.25,1.1) node {$3$};
\end{tikzpicture}
\caption{In the graphs shown above, $H$ is a subgraph of $G$, but it is not the induced subgraph on the vertex set $\{1,2,3\}$ since $H$ is missing the diagonal edge connecting vertices $1$ and $3$.}
\label{fig:induced graph}
\end{figure}

The following definitions, adapted from \cite{Bollobas-subgraph} and \cite{Analytic-bollobas}, will be used in sections \ref{sec:subgraphs}, \ref{sec:2torsion}, 
and \ref{sec:Betti numbers}.

\begin{defn}\label{mG}
The \textbf{essential density} of a graph $G$ is $$m(G):= \max\left\{\frac{|E(H)|}{|V(H)|}\; : \; H\subset G,\, |V(H)|>0\right\},$$
and $G$ is \textbf{strictly balanced} if $m(H)<m(G)$ for all proper subgraphs $H\subset G$.
\end{defn}

For a field $k$, a simplicial complex $\Delta$ on $n$ vertices has a corresponding Stanley--Reisner ideal $I_\Delta\subset S=k[x_1,\dots,x_n]$.  Since these $I_\Delta$ are squarefree monomial ideals, Hochster's formula~\cite[Theorem~5.5.1]{bruns-herzog} relates the Betti table of $S/I_\Delta$ to topological properties of $\Delta$, providing our key tool for studying this Betti table for various fields $k$.  
An immediate consequence of Hochster's formula is the following fact, which characterizes when these Betti tables are different over a field of characteristic $\ell$ than over $\Q$.
\begin{fact}\label{fact:depend}
For a simplicial complex $\Delta$, the Betti table of the Stanley--Reisner ideal $I_\Delta$ has $\ell$-torsion if and only if there exists a subset $\alpha \subset V(\Delta)$ such that $\Delta|_\alpha$ has $\ell$-torsion in one of its homology groups.
\end{fact}

\subsection{Monomial ideals from random flag complexes}\label{subsec:random flag}
Recall that a flag complex is a simplicial complex obtained from a graph by adjoining a $k$-simplex to every $(k+1)$-clique in the graph, which is called taking the clique complex. Therefore, a flag complex is entirely determined by its underlying graph. We write $\Delta \sim \Delta(n,p)$ to denote the flag complex which is the clique complex of an Erd\H{o}s--R\'enyi random graph $G(n,p)$ on $n$ vertices, where each edge is attached with probability $p$.  If $\alpha\subset V(\Delta)$, then we note that $\Delta|_{\alpha}$ is also flag.  The properties of random flag complexes have been analyzed extensively, with~\cite{kahle} providing an overview.  As discussed in the introduction, the syzygies of Stanley--Reisner ideals of random flag complexes were first studied in~\cite{erman-yang}.

\subsection{Probability}
We use the notation $\P[*]$ for the probability of an event.  If $X_n$ is a sequence of random variables, then we say that the event $X_n=x_0$ occurs {\bf with high probability as $n\to \infty$} if $\P[X_n = x_0]\to 1$ as $n\to \infty$. For a random variable $X$, we use $\E[X]$ for the expected value of $X$ and $\Var(X)$ for the variance of $X$.  

For functions $f(x)$ and $g(x)$, we write $f\ll g$ if $\ds\lim_{x\to \infty} f/g \to 0$.  We use $f\in O(g)$ if there is a constant $N$ where $|f(x)|\leq N|g(x)|$ for all sufficiently large values of $x$, and we use $f\in \Omega(g)$  if there is a constant $N'$ where $|f(x)|\geq N'|g(x)|$ for all sufficiently large values of $x$.

\section{Constructing a flag complex with $m$-torsion in homology}\label{sec:construction}
The goal of this section is to prove the following result:
\begin{theorem}\label{thm:Xm}
For every $m \geq 2$, there exists a two-dimensional flag complex $X_m$ such that the torsion subgroup of $H_1(X_m)$ is isomorphic to $\Z/m\Z$ and $\maxdeg(X_m)\leq 12$.
\end{theorem}
This result is the foundation of our proof of Theorem~\ref{thm:m torsion} as we will show that this specific complex $X_m$ appears as an induced subcomplex of $\Delta(n,p)$ with high probability as $n \to \infty$ under the hypotheses of that theorem.

Here is an overview of our proof of Theorem~\ref{thm:Xm}, which is largely based on ideas from~\cite{newman}.  Given an integer $m\geq 2$, we write its binary expansion as $m=2^{n_1}+\cdots+2^{n_k}$ with $0\leq n_1<\cdots<n_k$. Note that $k$ is the Hamming weight of $m$ and $n_k= \lfloor \log_2(m) \rfloor$. With this setup, the ``repeated squares presentation'' of $\Z/m\Z$ is given by
\[\Z/m\Z = \la \gamma_0,\gamma_1,\dots,\gamma_{n_k}\; |\; 2\gamma_0=\gamma_1, 2\gamma_1=\gamma_2,\dots, 2\gamma_{n_k-1}=\gamma_{n_k}, \gamma_{n_1}+\cdots+\gamma_{n_k}=0 \ra.\] We will construct a two-dimensional flag complex $X_m$ such that the torsion subgroup of $H_1(X_m)$ has this presentation. To do so, we follow Newman's ``telescope and sphere'' construction in \cite{newman}, where $Y_1$ is the telescope satisfying $$H_1(Y_1) \cong \la \gamma_0,\gamma_1,\dots,\gamma_{n_k}\; |\; 2\gamma_0=\gamma_1, 2\gamma_1=\gamma_2,\dots, 2\gamma_{n_k-1}=\gamma_{n_k} \ra,$$ $Y_2$ is the sphere satisfying $$H_1(Y_2) \cong \la \tau_1,\dots,\tau_k \; |\; \tau_1+\cdots+\tau_k=0 \ra,$$ and $X_m$ is created by gluing $Y_1$ and $Y_2$ together to yield a complex with the desired $H_1$-group.  Because we want our construction to be a flag complex with $\maxdeg(X_m)\leq 12$, we cannot simply quote Newman's results.  Instead, we must alter the triangulations to ensure that $Y_1$, $Y_2$, and $X_m$ are flag complexes.  Then, we must further alter the construction to reduce $\maxdeg(X_m)$. However, each of our constructions is homeomorphic to each of Newman's constructions.


\begin{notation}\label{notation:hamming etc}
Throughout the remainder of this section we assume that $m\geq 2$ is given.  We write $m=2^{n_1}+\cdots+2^{n_k}$ with $0\leq n_1<\cdots<n_k$. To simplify notation, we also denote $X_m$ by $X$ for the remainder of this section.
\end{notation}

\subsection{The telescope construction}
The telescope $Y_1$ that we construct will be homeomorphic to the $Y_1$ that Newman constructs in \cite[Proof of Lemma~3.1]{newman} for the $d=2$ case.  We start with building blocks which are punctured projective planes; in contrast with~\cite{newman}, our blocks are triangulated so that each is a flag complex. Explicitly, for each $i=0,\dots,(n_k-1)$, we produce a building block which is a triangulated projective plane with a square face removed, with vertices, edges, and faces as illustrated in Figure~\ref{fig:Y1}.  Our building blocks differ from Newman's in order to ensure that $Y_1$ and the final simplicial complex $X$ are flag complexes; for instance, we need to add extra vertices $v'_{8i},\dots,v'_{8i+7}$. 


\begin{figure}
\begin{center}
    \begin{tikzpicture}[scale=1]
    \draw [fill=lightgray, thick] (-1,3)--(-3,1)--(-1,2)--(-1,3);
    \draw [fill=lightgray, thick] (-3,1)--(-2,1)--(-1,2)--(-3,1);
    \draw [fill=lightgray, thick] (-2,1)--(-1,1)--(-1,2)--(-2,1);
    \draw [fill=lightgray, thick] (-3,1)--(-3,-1)--(-2,1)--(-3,1);
    \draw [fill=lightgray, thick] (-2,1)--(-3,-1)--(-2,-1)--(-2,1);
    \draw [fill=lightgray, thick] (-2,1)--(-2,-1)--(-1,-1)--(-2,1);
    \draw [fill=lightgray, thick] (-2,1)--(-1,-1)--(-1,1)--(-2,1);
    \draw [fill=lightgray, thick] (-3,-1)--(-1,-3)--(-2,-1)--(-3,-1);
    \draw [fill=lightgray, thick] (-2,-1)--(-1,-3)--(-1,-2)--(-2,-1);
    \draw [fill=lightgray, thick] (-2,-1)--(-1,-2)--(-1,-1)--(-2,-1);
    \draw [fill=lightgray, thick] (-1,-2)--(-1,-3)--(1,-3)--(-1,-2);
    \draw [fill=lightgray, thick] (-1,-2)--(1,-3)--(1,-2)--(-1,-2);
    \draw [fill=lightgray, thick] (-1,-2)--(1,-2)--(1,-1)--(-1,-2);
    \draw [fill=lightgray, thick] (-1,-2)--(1,-1)--(-1,-1)--(-1,-2);
    \draw [fill=lightgray, thick] (1,-2)--(1,-3)--(3,-1)--(1,-2);
    \draw [fill=lightgray, thick] (1,-2)--(3,-1)--(2,-1)--(1,-2);
    \draw [fill=lightgray, thick] (1,-2)--(2,-1)--(1,-1)--(1,-2);
    \draw [fill=lightgray, thick] (2,-1)--(3,-1)--(3,1)--(2,-1);
    \draw [fill=lightgray, thick] (2,-1)--(3,1)--(2,1)--(2,-1);
    \draw [fill=lightgray, thick] (2,-1)--(2,1)--(1,1)--(2,-1);
    \draw [fill=lightgray, thick] (2,-1)--(1,1)--(1,-1)--(2,-1);
    \draw [fill=lightgray, thick] (2,1)--(3,1)--(1,3)--(2,1);
    \draw [fill=lightgray, thick] (2,1)--(1,3)--(1,2)--(2,1);
    \draw [fill=lightgray, thick] (2,1)--(1,2)--(1,1)--(2,1);
    \draw [fill=lightgray, thick] (1,2)--(1,3)--(-1,3)--(1,2);
    \draw [fill=lightgray, thick] (1,2)--(-1,3)--(-1,2)--(1,2);
    \draw [fill=lightgray, thick] (1,2)--(-1,2)--(-1,1)--(1,2);
    \draw [fill=lightgray, thick] (1,2)--(-1,1)--(1,1)--(1,2);
    
    \node (0) at (-1,3) {};
    \node (0') at (1,-3) {};
    \node (1) at (-3,1) {};
    \node (1') at (3,-1) {};
    \node (2) at (-3,-1) {};
    \node (2') at (3,1) {};
    \node (3') at (1,3) {};
    \node (3) at (-1,-3) {};
    \node (4) at (-1,1) {};
    \node (5) at (-1,-1) {};
    \node (6) at (1,-1) {};
    \node (7) at (1,1) {};
    \node (8) at (-1,2) {};
    \node (9) at (-2,1) {};
    \node (10) at (-2,-1) {};
    \node (11) at (-1,-2) {};
    \node (12) at (1,-2) {};
    \node (13) at (2,-1) {};
    \node (14) at (2,1) {};
    \node (15) at (1,2) {};
    
    \node [above] at (-1,3) {$\scriptstyle v_{4i}$};
    \node [below] at (1,-3) {$\scriptstyle v_{4i}$};
    \node [left] at (-3,1) {$\scriptstyle v_{4i+1}$};
    \node [right] at (3,-1) {$\scriptstyle v_{4i+1}$};
    \node [left] at (-3,-1) {$\scriptstyle v_{4i+2}$};
    \node [right] at (3,1) {$\scriptstyle v_{4i+2}$};
    \node [above] at (1,3) {$\scriptstyle v_{4i+3}$};
    \node [below] at (-1,-3) {$\scriptstyle v_{4i+3}$};
    \node [below right] at (-1,1) {$\scriptstyle v_{4i+4}$};
    \node [above right] at (-1,-1) {$\scriptstyle v_{4i+5}$};
    \node [above left] at (1,-1) {$\scriptstyle v_{4i+6}$};
    \node [below left] at (1,1) {$\scriptstyle v_{4i+7}$};
    \node [above right] at (-1,2) {$\scriptstyle v'_{8i}$};
    \node [above left] at (-1.8,1) {$\scriptstyle v'_{8i+1}$};
    \node [above left] at (-2,-1) {$\scriptstyle v'_{8i+2}$};
    \node [below right] at (-1.05,-2.05) {$\scriptstyle v'_{8i+3}$};
    \node [below right] at (0.98,-1.98) {$\scriptstyle v'_{8i+4}$};
    \node [above right] at (2,-1.07) {$\scriptstyle v'_{8i+5}$};
    \node [above right] at (1.9, 0.93) {$\scriptstyle v'_{8i+6}$};
    \node [above] at (0.69,2.2) {$\scriptstyle v'_{8i+7}$};
    
    \draw[fill] (-1,3) circle [radius=2pt];
    \draw[fill] (1,-3) circle [radius=2pt];
    \draw[fill] (-3,1) circle [radius=2pt];
    \draw[fill] (3,-1) circle [radius=2pt];
    \draw[fill] (-3,-1) circle [radius=2pt];
    \draw[fill] (3,1) circle [radius=2pt];
    \draw[fill] (1,3) circle [radius=2pt];
    \draw[fill] (-1,-3) circle [radius=2pt];
    \draw[fill] (-1,1) circle [radius=2pt];
    \draw[fill] (-1,-1) circle [radius=2pt];
    \draw[fill] (1,-1) circle [radius=2pt];
    \draw[fill] (1,1) circle [radius=2pt];
    \draw[fill] (-1,2) circle [radius=2pt];
    \draw[fill] (-2,1) circle [radius=2pt];
    \draw[fill] (-2,-1) circle [radius=2pt];
    \draw[fill] (-1,-2) circle [radius=2pt];
    \draw[fill] (1,-2) circle [radius=2pt];
    \draw[fill] (2,-1) circle [radius=2pt];
    \draw[fill] (2,1) circle [radius=2pt];
    \draw[fill] (1,2) circle [radius=2pt];
    \end{tikzpicture}
\caption{Building block for the telescope construction with $i=0,1,\dots,(n_k-1)$.}
\label{fig:Y1}
\end{center}
\end{figure}
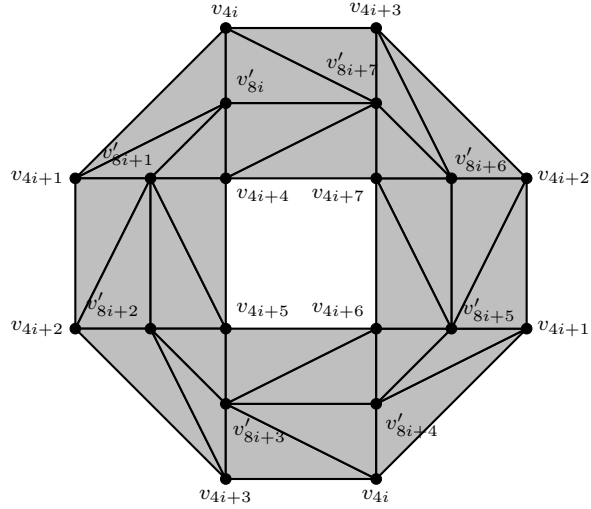
We construct $Y_1$ by identifying edges and vertices of these $n_k$ building blocks as labeled.  The underlying vertex set is $V(Y_1) = \{v_0,v_1,v_2,\dots,v_{4n_k+3},v'_0,v'_1,\dots,v'_{8n_k-1}\}$, so we have $|V(Y_1)|=12n_k+4$.  Since each building block has $44$ edges, $4$ of which are glued to the next building block, and $28$ faces, we have $|E(Y_1)|=40n_k+4$ and $|F(Y_1)|=28n_k$. In addition, observe that the vertices of highest degree are those in the squares in the ``middle'' of the telescope, such as vertex $v_4$ when $n_k\geq 2$. In this case, $v_4$ is adjacent to $v_5, v_7, v'_0, v'_1, v'_7, v'_8, v'_{15}, v'_{11},$ and $v'_{12}$, so $\deg(v_4)=9$. By the symmetry of $Y_1$, we have that $\maxdeg(Y_1)=9$ when $n_k\geq 2$, and $\maxdeg(Y_1)=6$ when $n_k=1$ (when $m=2$ or $3$).

To compute $H_1(Y_1)$, we simply apply the identical argument from~\cite{newman}. We order the vertices in the natural way, where $v_j>v_k$ if $j>k$, similarly for the $v_\ell'$, and where $v'_{\ell} > v_j$ for all $\ell,j$.  We let these vertex orderings induce orientations on the edges and faces of $Y_1$. For each $i=0,\dots,n_k$, denote by $\gamma_i$ the 1-cycle of $Y_1$ represented by $[v_{4i},v_{4i+1}]+[v_{4i+1},v_{4i+2}]+[v_{4i+2},v_{4i+3}]-[v_{4i},v_{4i+3}]$. Then $2\gamma_i - \gamma_{i+1}$ is a 1-boundary of $Y_1$ for each $i=0,\dots,(n_k-1)$, and, as in Newman's construction, we have that $H_1(Y_1)$ can be presented as $\la \gamma_0,\gamma_1,\dots,\gamma_{n_k}\; |\; 2\gamma_0=\gamma_1, 2\gamma_1=\gamma_2,\dots, 2\gamma_{n_k-1}=\gamma_{n_k} \ra$.

\subsection{The sphere construction}
The sphere part $Y_2$ is a flag triangulation of the sphere $S^2$ that has $k$ square holes such that the squares are all vertex disjoint and nonadjacent. Our $Y_2$ will be homeomorphic to the $Y_2$ that Newman constructs in \cite{newman} for the $d=2$ case, but our construction involves a few different steps. First, we will show that for any integer $k\geq 1$, there exists a flag triangulation $T_i$ of $S^2$ (here $i=\lfloor \frac{k-1}{4}\rfloor$) with at least $k$ faces such that $\maxdeg(T_i)\leq 6$. Then, we will insert square holes on $k$ of the faces of $T_i$, while subdividing the edges, and call the resulting flag complex $\widetilde{T_i}$. Finally, we describe a process to replace each vertex of degree 14 in $\widetilde{T_i}$ with two degree 9 vertices so that the resulting complex, $Y_2$, has $\maxdeg(Y_2)\leq 12$. Throughout these constructions, we will have four cases corresponding to the value of $k \mod 4$, and we carefully keep track of the degrees of each vertex in $T_i$, $\widetilde{T_i}$, and $Y_2$ for each case.

\subsubsection{$T_i$ and flag bistellar 0-moves}
We begin by constructing an infinite sequence $T_0, T_1, \dots$ of flag triangulations of $S^2$ such that $\maxdeg(T_i)\leq 6$ for all $i$. To do so, we adapt the bistellar 0-moves used in \cite[Lemma~5.6]{newman}. Let $T_0$ be the $3$-simplex boundary on the vertex set $\{w_0, w_1, w_2, w_3\}$. Note that each vertex of $T_0$ has degree 3. We will construct the remaining $T_i$ inductively. To build $T_1$, first remove the face $[w_1, w_2, w_3]$ and edge $[w_1, w_3]$. Then, add two new vertices $w_4$ and $w_5$ as well as new edges $[w_0, w_4], [w_1, w_4], [w_3, w_4], [w_1, w_5],$ $[w_2, w_5], [w_3, w_5],$ and $[w_4, w_5]$. Taking the clique complex will then give $T_1$.  See Figure~\ref{fig:triangulations}.

Essentially, this process is the same as making the face $[w_1,w_2,w_3]$ into a square face $[w_1,w_2,w_3,w_4]$, removing that square face, taking the cone over it, and then ensuring that the resulting complex is a flag triangulation of $S^2$. We will call such a move a \textbf{flag bistellar 0-move}. Each $T_{i+1}$ for $i\geq 0$ will be obtained from $T_i$ by performing a flag bistellar 0-move on the face $[w_{2i+1}, w_{2i+2}, w_{2i+3}]$ of $T_i$. Explicitly, to construct $T_{i+1}$, remove the face $[w_{2i+1}, w_{2i+2}, w_{2i+3}]$ and the edge $[w_{2i+1}, w_{2i+3}]$. Then, add new vertices $w_{2i+4}$ and $w_{2i+5}$ and new edges $[w_{2i}, w_{2i+4}], [w_{2i+1}, w_{2i+4}], [w_{2i+3}, w_{2i+4}], [w_{2i+1},w_{2i+5}],$ $[w_{2i+2}, w_{2i+5}], [w_{2i+3}, w_{2i+5}],$ $[w_{2i+4}, w_{2i+5}],$ and take the clique complex to get $T_{i+1}$. Note that each flag bistellar 0-move adds 2 vertices, 6 edges, and 4 faces. Since $|V(T_0)|=4, |E(T_0)|=6$, and $|F(T_0)|=4$, this means that $|V(T_i)|=2i+4$, $|E(T_i)|=6i+6$, and $|F(T_i)|=4i+4$.

Further, Table \ref{tab:Ti_vertices} summarizes the degrees of the vertices in each $T_i$.  
\begin{table}[ht]
\begin{center}
    \begin{tabular}{|c|c|c|}
    \hline
    $T_i$ & Degree & Vertices \\
    \hline
    $T_0$ & 3 & $w_0, w_1, w_2, w_3$ \\
    \hline
    $T_1$ & 4 & $w_0, w_1, w_2, w_3, w_5, w_6$ \\
    \hline
    $T_2$ & 4 &  $w_0, w_1, w_6, w_7$\\
     & 5 & $w_2, w_3, w_4, w_5$\\ 
     \hline
    $T_i$ & 4 &  $w_0, w_1, w_{2i+2}, w_{2i+3}$\\
    $i\geq 3$ & 5 & $w_2, w_3, w_{2i}, w_{2i+1}$\\
    & 6 & $w_4, \dots, w_{2i-1}$\\
    \hline
    \end{tabular}
\end{center}
\caption{Degrees of the vertices in $T_i$.}
\label{tab:Ti_vertices}
\end{table}
To compute the degrees of vertices in $T_i$ for $i\geq 3$, observe that when the new vertices $w_{2i+2}$ and $w_{2i+3}$ are added, they have degree $4$ in $T_i$. For each of the next two iterations of the flag bistellar-0 move, the degree of these vertices increases by one, resulting in degree 6 in $T_{i+2}$. In the remaining triangulations $T_j$ with $j\geq i+3$, these vertices are not affected. Therefore, $\maxdeg(T_i)\leq 6$ for each $i$.

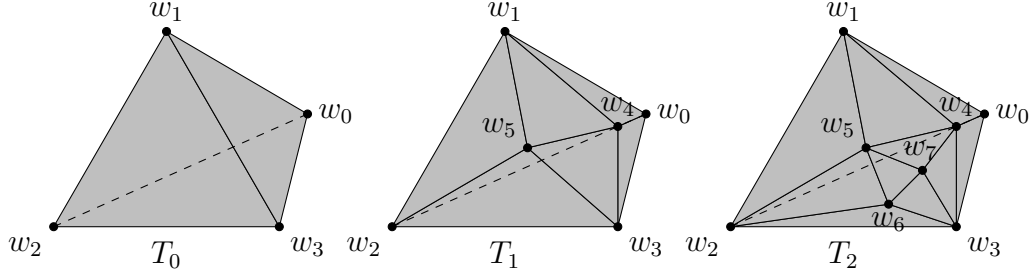
\begin{figure}
    \centering
    \begin{tikzpicture}[scale=1.5]
    \draw [fill=lightgray] (-1,0)--(1,0)--(0,1.732)--(-1,0);
    \draw [fill=lightgray] (1,0)--(1.25,1)--(0,1.732)--(1,0);
    \draw [dashed] (-1,0)--(1.25,1);
    
    \draw[fill] (-1,0) circle [radius=1pt];
    \draw[fill] (1,0) circle [radius=1pt];
    \draw[fill] (0,1.732) circle [radius=1pt];
    \draw[fill] (1.25,1) circle [radius=1pt];

    \node [above] at (0,1.732) {$w_1$};
    \node [below left] at (-1,0) {$w_2$};
    \node [below right] at (1,0) {$w_3$};
    \node [right] at (1.25,1) {$w_0$};
    \node at (0, -0.25) {$T_0$};
    
    \draw [fill=lightgray] (2,0)--(4,0)--(3.2,0.7)--(2,0);
    \draw [fill=lightgray] (2,0)--(3.2,0.7)--(3,1.732)--(2,0);
    \draw [fill=lightgray] (4,0)--(4.25,1)--(4,0.889)--(4,0);
    \draw [fill=lightgray] (4,0)--(4,0.889)--(3.2,0.7)--(4,0);
    \draw [fill=lightgray] (3.2,0.7)--(4,0.889)--(3,1.732)--(3.2,0.7);
    \draw [fill=lightgray] (4,0.889)--(4.25,1)--(3,1.732)--(4,0.889);
    \draw [dashed] (2,0)--(4.25,1);
    
    \draw[fill] (2,0) circle [radius=1pt];
    \draw[fill] (4,0) circle [radius=1pt];
    \draw[fill] (3,1.732) circle [radius=1pt];
    \draw[fill] (4.25,1) circle [radius=1pt];
    \draw[fill] (3.2,0.7) circle [radius=1pt];
    \draw[fill] (4,0.889) circle [radius=1pt];

    \node [above] at (3,1.732) {$w_1$};
    \node [below left] at (2,0) {$w_2$};
    \node [below right] at (4,0) {$w_3$};
    \node [right] at (4.25,1) {$w_0$};
    \node [above left] at (3.2,0.7) {$w_5$};
    \node [above] at (4,0.889) {$w_4$};
    \node at (3, -0.25) {$T_1$};
    
    \draw [fill=lightgray] (5,0)--(7,0)--(6.4,0.2)--(5,0);
    \draw [fill=lightgray] (5,0)--(6.4,0.2)--(6.2,0.7)--(5,0);
    \draw [fill=lightgray] (5,0)--(6.2,0.7)--(6,1.732)--(5,0);
    \draw [fill=lightgray] (7,0)--(7.25,1)--(7,0.889)--(7,0);
    \draw [fill=lightgray] (6.2,0.7)--(7,0.889)--(6,1.732)--(6.2,0.7);
    \draw [fill=lightgray] (7,0.889)--(7.25,1)--(6,1.732)--(7,0.889);
    \draw [fill=lightgray] (7,0)--(6.7,0.5)--(6.4,0.2)--(7,0);
    \draw [fill=lightgray] (6.4,0.2)--(6.7, 0.5)--(6.2, 0.7)--(6.4, 0.2);
    \draw [fill=lightgray] (7,0)--(7,0.889)--(6.7, 0.5)--(7,0);
    \draw [fill=lightgray] (6.7,0.5)--(7,0.889)--(6.2,0.7)--(6.7,0.5);
    \draw [dashed] (5,0)--(7.25,1);
    
    \draw[fill] (5,0) circle [radius=1pt];
    \draw[fill] (7,0) circle [radius=1pt];
    \draw[fill] (6,1.732) circle [radius=1pt];
    \draw[fill] (7.25,1) circle [radius=1pt];
    \draw[fill] (6.2,0.7) circle [radius=1pt];
    \draw[fill] (7,0.889) circle [radius=1pt];
    \draw[fill] (6.4,0.2) circle [radius=1pt];
    \draw[fill] (6.7, 0.5) circle [radius=1pt];

    \node [above] at (6,1.732) {$w_1$};
    \node [below left] at (5,0) {$w_2$};
    \node [below right] at (7,0) {$w_3$};
    \node [right] at (7.25,1) {$w_0$};
    \node [above left] at (6.2,0.7) {$w_5$};
    \node [above] at (7,0.889) {$w_4$};
    \node [below] at (6.4, 0.22) {$w_6$};
    \node [above] at (6.68,0.5) {$w_7$};
    \node at (6, -0.25) {$T_2$};
    
    \end{tikzpicture}
    \caption{The first few flag triangulations of $S^2$ using flag bistellar 0-moves.}
    \label{fig:triangulations}
\end{figure}

From this infinite sequence of flag triangulations of $S^2$ with bounded degree, we are interested in the particular $T_i$ with $i=\lfloor \frac{k-1}{4} \rfloor$ to use in our construction of $Y_2$, where $k$ is the Hamming weight of $m$ as in Notation~\ref{notation:hamming etc}. Note that this $T_i$ has vertex set $\{w_0,\dots,w_{2i+3}\}$ and has $4\lfloor \frac{k-1}{4}\rfloor+4$ faces. 
Let $\delta$ be the integer $0\leq \delta \leq 3$ where $\delta \equiv -k \mod 4$. Then $T_i$ has exactly $k+\delta$ faces.

\subsubsection{Constructing $\widetilde{T_i}$}
Next, we insert square holes in the first $k$ faces of $T_i$ and subdivide the remaining faces in such a way that the squares will be vertex disjoint and nonadjacent. 

First, we will insert square holes in $k$ of the faces of $T_i$, making sure to triangulate the resulting faces and take the clique complex so that our simplicial complex remains flag. Let $[w_r,w_s,w_t]$ with $r<s<t$ be the $j$th of these $k$ faces with respect to a fixed ordering of the faces (where $j$ ranges from 1 to $k$). We remove this face and subdivide the edges by adding new vertices $w'_{r,s}, w'_{r,t},$ and $w'_{s,t}$ and new edges $[w_r, w'_{r,s}], [w_s, w'_{r,s}], [w_r, w'_{r,t}], [w_t, w'_{r,t}],$ $[w_s, w'_{s,t}],$ and $[w_t, w'_{s,t}]$. Then, we add vertices $u_{4j-4}, u_{4j-3}, u_{4j-2},$ and $u_{4j-1}$ to form a square inside the original face with indices increasing counterclockwise. Moreover, we add edges
\begin{align*}
    & [w_r, u_{4j-4}], [w_r, u_{4j-1}], [u_{4j-4}, w'_{r,s}], [u_{4j-3}, w'_{r,s}], [w_s, u_{4j-3}] \\
    & [u_{4j-3}, w'_{s,t}], [u_{4j-2}, w'_{s,t}], [w_t, u_{4j-2}], [u_{4j-2}, w'_{r,t}], [u_{4j-1}, w'_{r,t}].
\end{align*}
\noindent After applying this process,
we take the clique complex. The result of this operation on face $[w_r, w_s, w_t]$ is depicted in Figure \ref{fig:Y2_faces} (left).

The remaining $\delta$ faces of $T_i$ will simply be subdivided and triangulated before taking the clique complex. Explicitly, this means that after removing the face $[w_{2i+1},w_{2i+2},w_{2i+3}]$ and its edges, we add vertices $w'_{2i+1,2i+2}, w'_{2i+1,2i+3},$ and $w'_{2i+2,2i+3}$ and edges 
\begin{align*}
&[w_{2i+1}, w'_{2i+1,2i+2}], [w_{2i+2}, w'_{2i+1,2i+2}], [w_{2i+1}, w'_{2i+1,2i+3}],\\
&[w_{2i+3}, w'_{2i+1,2i+3}], [w'_{2i+1,2i+2}, w'_{2i+1,2i+3}], [w_{2i+2}, w'_{2i+2,2i+3}],\\
&[w_{2i+3}, w'_{2i+2,2i+3}], [w'_{2i+1,2i+2}, w'_{2i+2,2i+2}], [w'_{2i+1,2i+3}, w'_{2i+2,2i+3}].
\end{align*}
This subdivision of face $[w_{2i+1},w_{2i+2},w_{2i+3}]$ is shown in Figure \ref{fig:Y2_faces} (right). We do similarly for the faces $[w_{2i-1},w_{2i+2}, w_{2i+3}]$ and $[w_{2i}, w_{2i+1}, w_{2i+3}]$, if necessary. The clique complex of this construction is a flag complex which is homeomorphic to $S^2$ with $k$ distinct points removed. Call this complex $\widetilde{T_i}$.

\begin{figure}
    \centering
    \begin{tikzpicture}[scale=2.8, decoration={markings,mark=at position 0.5 with {\arrow{latex}}}]
    \draw [fill=lightgray] (-1,0)--(0,0)--(-0.2,0.667)--(-1,0);
    \draw [fill=lightgray] (-1,0)--(-0.2,0.667)--(-0.5,0.866)--(-1,0);
    \draw [fill=lightgray] (-0.5,0.866)--(-0.2,0.667)--(-0.2,1.066)--(-0.5,0.866);
    \draw [fill=lightgray] (-0.5,0.866)--(-0.2,1.066)--(0,1.732)--(-0.5,0.866);
    \draw [fill=lightgray] (0,1.732)--(-0.2,1.066)--(0.2,1.066)--(0,1.732);
    \draw [fill=lightgray] (0.5,0.866)--(0.2,1.066)--(0,1.732)--(0.5,0.866);
    \draw [fill=lightgray] (0.5,0.866)--(0.2,0.667)--(0.2,1.066)--(0.5,0.866);
    \draw [fill=lightgray] (1,0)--(0.2,0.667)--(0.5,0.866)--(1,0);
    \draw [fill=lightgray] (1,0)--(0,0)--(0.2,0.667)--(1,0);
    \draw [fill=lightgray] (0,0)--(-0.2,0.667)--(0.2,0.667)--(0,0);
    
    \draw[postaction={decorate}, thick] (0,1.723)--(-0.5,0.866);
    \draw[postaction={decorate}, thick] (0,1.723)--(-0.2,1.066);
    \draw[postaction={decorate}, thick] (-0.2,1.066)--(-0.5,0.866);
    \draw[postaction={decorate}, thick] (-0.2,0.667)--(-0.5,0.866);
    \draw[postaction={decorate}, thick] (-0.2,1.066)--(-0.2,0.667);
    \draw[postaction={decorate}, thick] (-1,0)--(-0.2,0.667);
    \draw[postaction={decorate}, thick] (-1,0)--(-0.5,0.866);
    \draw[postaction={decorate}, thick] (-1,0)--(0,0);
    \draw[postaction={decorate}, thick] (-0.2,0.667)--(0,0);
    \draw[postaction={decorate}, thick] (1,0)--(0,0);
    \draw[postaction={decorate}, thick] (0.2,0.667)--(0,0);
    \draw[postaction={decorate}, thick] (1,0)--(0.2,0.667);
    \draw[postaction={decorate}, thick] (1,0)--(0.5, 0.866);
    \draw[postaction={decorate}, thick] (0.2,0.667)--(0.5, 0.866);
    \draw[postaction={decorate}, thick] (0.2,1.066)--(0.5, 0.866);
    \draw[postaction={decorate}, thick] (0,1.723)--(0.2,1.066);
    \draw[postaction={decorate}, thick] (0,1.723)--(0.5, 0.866);
    \draw[postaction={decorate}, thick] (0.2,0.667)--(0.2,1.066);
    \draw[postaction={decorate}, thick] (-0.2,1.066)--(0.2,1.066);
    \draw[postaction={decorate}, thick] (-0.2,0.667)--(0.2,0.667);

    \draw [fill=lightgray] (2,0)--(3,0)--(2.5,0.866)--(2,0);
    \draw [fill=lightgray] (3,0)--(2.5,0.866)--(3.5,0.866)--(3,0);
    \draw [fill=lightgray] (4,0)--(3.5,0.866)--(3,0)--(4,0);
    \draw [fill=lightgray] (2.5,0.866)--(3.5,0.866)--(3,1.732)--(2.5,0.866);
    
    \draw[postaction={decorate}, thick] (3,1.723)--(2.5,0.866);
    \draw[postaction={decorate}, thick] (2,0)--(2.5,0.866);
    \draw[postaction={decorate}, thick] (2,0)--(3,0);
    \draw[postaction={decorate}, thick] (4,0)--(3,0);
    \draw[postaction={decorate}, thick] (3.5, 0.866)--(3,0);
    \draw[postaction={decorate}, thick] (3,1.723)--(3.5, 0.866);
    \draw[postaction={decorate}, thick] (2.5, 0.866)--(3,0);
    \draw[postaction={decorate}, thick] (4, 0)--(3.5,0.866);
    \draw[postaction={decorate}, thick] (2.5, 0.866)--(3.5,0.866);
    
    \draw[fill] (-1,0) circle [radius=1pt];
    \draw[fill] (0,0) circle [radius=1pt];
    \draw[fill] (1,0) circle [radius=1pt];
    \draw[fill] (0,1.732) circle [radius=1pt];
    \draw[fill] (-0.5,0.866) circle [radius=1pt];
    \draw[fill] (0.5,0.866) circle [radius=1pt];
    \draw[fill] (-0.2,1.066) circle [radius=1pt];
    \draw[fill] (0.2,1.066) circle [radius=1pt];
    \draw[fill] (-0.2,0.667) circle [radius=1pt];
    \draw[fill] (0.2,0.667) circle [radius=1pt];
    
    \draw[fill] (2,0) circle [radius=1pt];
    \draw[fill] (3,0) circle [radius=1pt];
    \draw[fill] (4,0) circle [radius=1pt];
    \draw[fill] (3,1.732) circle [radius=1pt];
    \draw[fill] (2.5,0.866) circle [radius=1pt];
    \draw[fill] (3.5,0.866) circle [radius=1pt];
    
    \node [above] at (0,1.732) {$w_r$};
    \node [below left] at (-1,0) {$w_s$};
    \node [below right] at (1,0) {$w_t$};
    \node [left] at (-0.5,0.866) {$w'_{r,s}$};
    \node [below] at (0,0) {$w'_{s,t}$};
    \node [right] at (0.5,0.866) {$w'_{r,t}$};
    \node [above right] at (-0.2,1.06) {$\scriptstyle u_{4j-4}$};
    \node [below right] at (-0.2,0.68) {$\scriptstyle u_{4j-3}$};
    \node [above left] at (0.22,0.63) {$\scriptstyle u_{4j-2}$};
    \node [below left] at (0.22,1.08) {$\scriptstyle u_{4j-1}$};
    
    \node [above] at (3,1.732) {$w_{2i+1}$};
    \node [below left] at (2,0) {$w_{2i+2}$};
    \node [below right] at (4,0) {$w_{2i+3}$};
    \node [left] at (2.5,0.866) {$w'_{2i+1,2i+2}$};
    \node [below] at (3,0) {$w'_{2i+2,2i+3}$};
    \node [right] at (3.5,0.866) {$w'_{2i+1,2i+3}$};
    
    \end{tikzpicture}
    \caption{Example of square insertion done on $k$ faces of $T_i$ (left), and subdivided triangulation on remaining faces (right).}
    \label{fig:Y2_faces}
\end{figure}
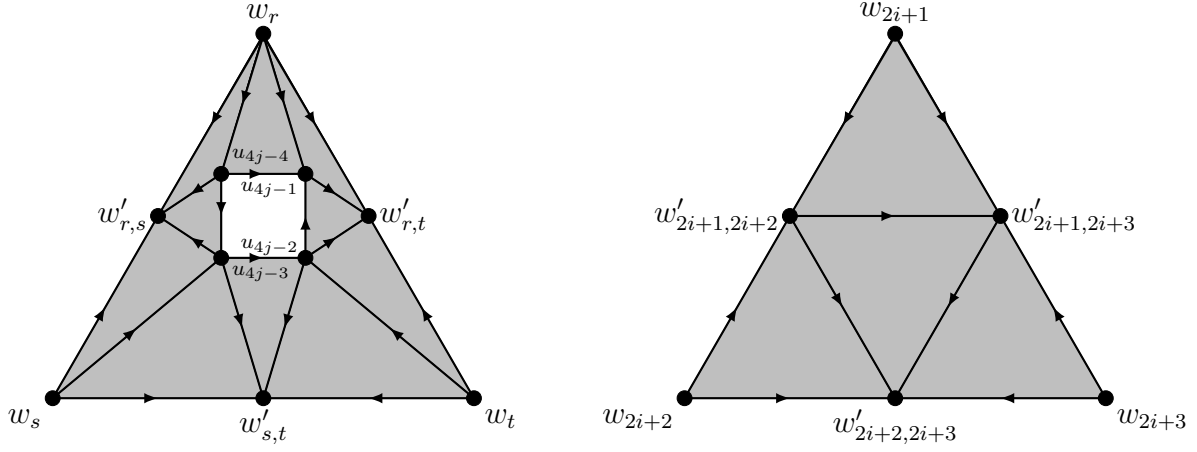

Let's consider the degrees of the vertices of $\widetilde{T_i}$. We have that $\deg(w'_{s,t})=6$ for all $s,t$ and $\deg(u_\ell)\in \{4,5\}$ for all $\ell$, where the ``top'' $u_\ell$ have degree 4 and the ``bottom'' $u_\ell$ have degree 5. To determine the degrees of the $w_j$ vertices, we need to consider their degrees in $T_i$ and how their degrees increase during the subdivision and square face removal processes. As we are interested in bounding the maximum degree of the vertices of $\widetilde{T_i}$, we need only consider the case when $\delta=0$ and all $k$ faces of $T_i$ have a square hole.
\begin{table}[ht]
\begin{center}
    \begin{tabular}{|c|c|c|}
    \hline
    $\widetilde{T_i}$ & Degree & Vertices \\
    \hline
     & 6 & $w_2, w_3$ \\
    $\td{T_0}$ & 7 & $w_1$\\
    $(k=4)$& 9 & $w_0$\\
    \hline
     & 8 & $w_4, w_5$\\ 
    $\td{T_1}$ & 9 & $w_2, w_3$ \\
    $(k=8)$& 10 & $w_1$\\
    & 12 & $w_0$\\
    \hline
    & 8 & $w_6, w_7$\\ 
    $\td{T_2}$ & 10 & $w_1$ \\
    $(k=12)$ & 11 & $w_4, w_5$\\
    & 12 & $w_0, w_2, w_3$\\
    \hline
     & 8 & $w_{2i+2}, w_{2i+3}$\\
    $\widetilde{T_i}$ & 10 & $w_1$ \\
    $i\geq 3$ & 11 & $w_{2i}, w_{2i+1}$\\
    $(k=4i+4)$& 12 & $w_0, w_2, w_3$\\
    & 14 & $w_4, \dots, w_{2i-1}$\\
    \hline
    \end{tabular}
\end{center}
\caption{Degrees of the vertices in $\widetilde{T_i}$ when $k\equiv 0 \mod 4$.}
\label{tab:Ti_tilde}
\end{table}
Table \ref{tab:Ti_tilde} gives the degrees of each of the $w_j$ vertices in $\widetilde{T_i}$ when $\delta=0$.

To verify the degrees of the $w_j$ in $\widetilde{T_i}$ when $i\geq 3$, we consider how the degrees of the vertices change as $i$ increases. Between $\td{T}_{i-1}$ and $\widetilde{T_i}$ (with $\delta=0$ for both), the only vertices that change degree are $w_{2i-2}, w_{2i-1}, w_{2i}, w_{2i+1}$, each of which increase degree by 3. This is because they each get one new edge from the $T_i$ flag bistellar 0-move and two new edges from the square removal triangulation process (since each vertex is the smallest indexed and hence the ``top'' vertex of one new triangular face). Further, the new vertices $w_{2i+2}, w_{2i+3}$ in $\widetilde{T_i}$ have degree 8, and they increase degree by 3 in the next two iterations, resulting in degree 14 in $\td{T}_{i+2}$ and all future iterations.

The above argument shows that regardless of $m$ and $k$, $\maxdeg(\widetilde{T_i})\leq 14$, where $i=\lfloor \frac{k-1}{4} \rfloor$. Furthermore, the only vertices that could have degree 14 are $w_4,\dots, w_{2i-1}$, each of which is separated from the others by a $w'_{s,t}$ vertex, which only has degree 6. We want to know exactly which vertices in $\widetilde{T_i}$ have degree 14, for all possible $k$ with $i\geq 3$, because we plan to alter these vertices to decrease $\maxdeg(\widetilde{T_i})$. Note that as $\delta$ increases from 0 to 3, the degree of each $w_j$ vertex is nonincreasing. When $k=4i+4$ and $\delta=0$, Table \ref{tab:Ti_tilde} gives that $w_4,\dots, w_{2i-1}$ have degree 14. When $k=4i+3$ and $\delta=1$, the face $[w_{2i+1},w_{2i+2},w_{2i+3}]$ is subdivided instead of having a square removed, but this does not change the degrees of $w_4,\dots, w_{2i-1}$, so these all still have degree 14. When $k=4i+2$ and $\delta=2$, the faces $[w_{2i+1},w_{2i+2},w_{2i+3}]$ and $[w_{2i-1},w_{2i+2},w_{2i+3}]$ are subdivided. Therefore, $w_{2i-1}$ has two fewer edges than in the previous case since $w_{2i-1}$ is the smallest indexed vertex in $[w_{2i-1}, w_{2i+2},w_{2i+3}]$ and so would have two ``top'' $u_\ell$ adjacent to it if this face had a square removed from it. So, in this case, $w_4,\dots, w_{2i-2}$ have degree 14 and $w_0,w_2, w_3, w_{2i-1}$ have degree 12 in $\widetilde{T_i}$. Finally, if $k=4i+1$ and $\delta=3$, then additionally the face $[w_{2i},w_{2i+1},w_{2i+3}]$ is subdivided, which means that the degree 12 and 14 vertices are the same as in the previous case.

\subsubsection{Replacing degree 14 vertices to construct $Y_2$}
Having identified the vertices of $\widetilde{T_i}$ of the highest degree, we now describe a process by which we will replace each vertex of degree 14 by two vertices of degree 9 in order to ensure that $\maxdeg(\widetilde{T_i})\leq 12$ for all $k$ (and $i$). The resulting flag complex, given by taking the clique complex of this construction, will be the final $Y_2$, and it will be homeomorphic to $\widetilde{T_i}$.  The process is summarized by Figure~\ref{fig:Replacing_vertex} and described in detail in the following paragraphs. 

Suppose $w_j$ is a vertex of degree 14 in $\widetilde{T_i}$. Locally, on a small neighborhood of $w_j$, $\widetilde{T_i}$ is homeomorphic to a $2$-manifold. Since $\deg(w_j)=14$, $w_j$ is surrounded by six triangular faces coming from $T_i$, all of which have had a square removed. By our construction, two of these squares (which are in adjacent triangular faces) have both of their ``top'' $u_\ell$ vertices connected to $w_j$, but the other four squares just have a single edge connecting one of their ``bottom'' $u_\ell$ vertices to $w_j$. So, $w_j$ has six $w'_{s,t}$ neighbors and eight $u_\ell$ neighbors, which form a 14-sided polygon with $w_j$ as its ``star'' point. Choose two $w'_{s,t}$ vertices which are across from each other in this 14-sided polygon, say $w'_{a,b}$ and $w'_{c,d}$. Next, we will remove $w_j$ and all of the 14 faces that it is contained in. Then, we add vertices $w_{j_1}$ and $w_{j_2}$ in place of $w_j$ and add edges in such a way that $\deg(w_{j_1})=\deg(w_{j_2})=9$, there are edges $[w_{j_1},w_{j_2}], [w_{j_1},w'_{a,b}], [w_{j_1},w'_{c,d}], [w_{j_2},w'_{a,b}],$ and $[w_{j_2},w'_{c,d}]$, and the 14-sided polygon is triangulated with 16 triangles. This process only changes the degree of $w'_{a,b}$ and $w'_{c,d}$, each of which now have degree 7. Therefore, the maximum degree of $w_{j_1}, w_{j_2}$, and the 14 vertices in the polygon is 9 (since $\deg(u_\ell)\in \{4,5\}$ and $\deg(w'_{s,t})=6$). To illustrate this construction, we consider the case when $k=20$. Then $i=4$, $\delta=0$, and $\deg(w_7)=14$ in $\td{T_4}$. Figure \ref{fig:Replacing_vertex} depicts this process when $w'_{a,b}=w'_{3,7}$ and $w'_{c,d}=w'_{7,11}$.

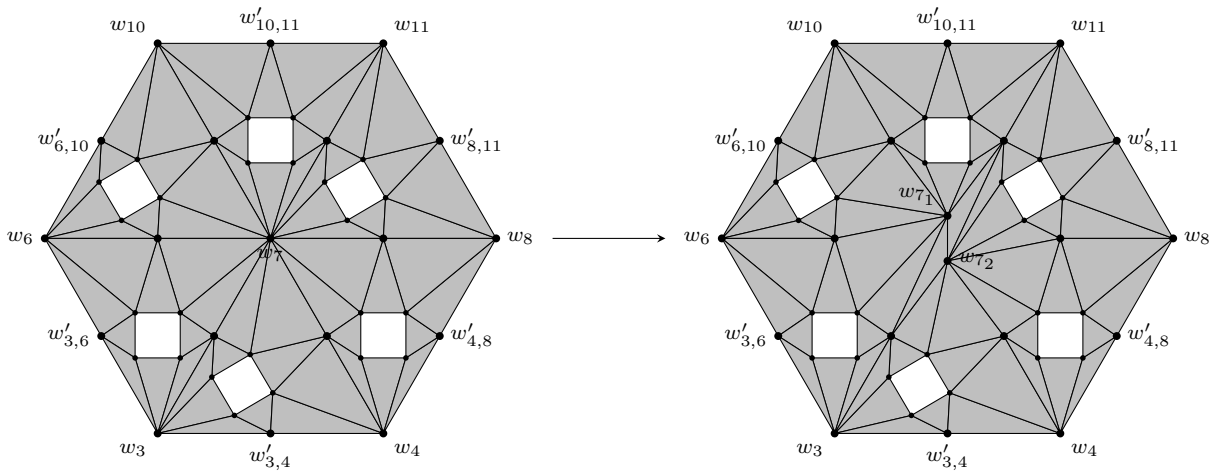
\begin{figure}
\begin{center}
\begin{tikzpicture}[scale=1.5]
    \draw [fill=lightgray] (0,0)--(0.2,0.67)--(-0.2,0.67)--(0,0);
    \draw [fill=lightgray] (0,0)--(0.2,0.67)--(0.5,0.866)--(0,0);
    \draw [fill=lightgray] (0.2,0.67)--(0.5,0.866)--(0.2, 1.07)--(0.2,0.67);
    \draw [fill=lightgray] (0.5,0.866)--(0.2, 1.07)--(1,1.73)--(0.5,0.866);
    \draw [fill=lightgray] (0.2, 1.07)--(1,1.73)--(0,1.73)--(0.2,1.07);
    \draw [fill=lightgray] (0,1.73)--(0.2,1.07)--(-0.2,1.07)--(0,1.73);
    \draw [fill=lightgray] (-0.2,1.07)--(0,1.73)--(-1,1.73)--(-0.2,1.07);
    \draw [fill=lightgray] (-1,1.73)--(-0.2,1.07)--(-0.5,0.866)--(-1,1.73);
    \draw [fill=lightgray] (-0.2,1.07)--(-0.5,0.866)--(-0.2,0.67)--(-0.2,1.07);
    \draw [fill=lightgray] (-0.5,0.866)--(-0.2,0.67)--(0,0)--(-0.5, 0.866);
    \draw [fill=lightgray] (0,0)--(0.68, 0.16)--(0.48, 0.5)--(0,0);
    \draw [fill=lightgray] (0,0)--(0.68, 0.16)--(1,0)--(0,0);
    \draw [fill=lightgray] (0.68, 0.16)--(1,0)--(1.02, 0.36)--(0.68, 0.16);
    \draw [fill=lightgray] (1,0)--(1.02, 0.36)--(2,0)--(1,0);
    \draw [fill=lightgray] (1.02, 0.36)--(2,0)--(1.5,0.866)--(1.02,0.36);
    \draw [fill=lightgray] (1.5,0.866)--(1.02,0.36)--(0.82,0.7)--(1.5,0.866);
    \draw [fill=lightgray] (0.82,0.7)--(1.5,0.866)--(1,1.73)--(0.82,0.7);
    \draw [fill=lightgray] (1,1.73)--(0.82,0.7)--(0.5,0.866)--(1,1.73);
    \draw [fill=lightgray] (0.82,0.7)--(0.5,0.866)--(0.48, 0.5)--(0.82,0.7);
    \draw [fill=lightgray] (0.5,0.866)--(0.48, 0.5)--(0,0)--(0.5, 0.866);
    \draw [fill=lightgray] (-2,0)--(-1.32, 0.16)--(-1.52, 0.5)--(-2,0);
    \draw [fill=lightgray] (-2,0)--(-1.32, 0.16)--(-1,0)--(-2,0);
    \draw [fill=lightgray] (-1.32, 0.16)--(-1,0)--(-0.98, 0.36)--(-1.32, 0.16);
    \draw [fill=lightgray] (-1,0)--(-0.98, 0.36)--(0,0)--(-1,0);
    \draw [fill=lightgray] (-0.98, 0.36)--(0,0)--(-0.5,0.866)--(-0.98,0.36);
    \draw [fill=lightgray] (-0.5,0.866)--(-0.98,0.36)--(-1.18,0.7)--(-0.5,0.866);
    \draw [fill=lightgray] (-1.18,0.7)--(-0.5,0.866)--(-1,1.73)--(-1.18,0.7);
    \draw [fill=lightgray] (-1,1.73)--(-1.18,0.7)--(-1.5,0.866)--(-1,1.73);
    \draw [fill=lightgray] (-1.18,0.7)--(-1.5,0.866)--(-1.52, 0.5)--(-1.18,0.7);
    \draw [fill=lightgray] (-1.5,0.866)--(-1.52, 0.5)--(-2,0)--(-1.5, 0.866);
    \draw [fill=lightgray] (-1,-1.73)--(-0.8,-1.06)--(-1.2,-1.06)--(-1,-1.73);
    \draw [fill=lightgray] (-1,-1.73)--(-0.8,-1.06)--(-0.5,-0.866)--(-1,-1.73);
    \draw [fill=lightgray] (-0.8,-1.06)--(-0.5,-0.866)--(-0.8, -0.66)--(-0.8,-1.06);
    \draw [fill=lightgray] (-0.5,-0.866)--(-0.8, -0.66)--(0,0)--(-0.5,-0.866);
    \draw [fill=lightgray] (-0.8, -0.66)--(0,0)--(-1,0)--(-0.8,-0.66);
    \draw [fill=lightgray] (-1,0)--(-0.8,-0.66)--(-1.2,-0.66)--(-1,0);
    \draw [fill=lightgray] (-1.2,-0.66)--(-1,0)--(-2,0)--(-1.2,-0.66);
    \draw [fill=lightgray] (-2,0)--(-1.2,-0.66)--(-1.5,-0.866)--(-2,0);
    \draw [fill=lightgray] (-1.2,-0.66)--(-1.5,-0.866)--(-1.2,-1.06)--(-1.2,-0.66);
    \draw [fill=lightgray] (-1.5,-0.866)--(-1.2,-1.06)--(-1,-1.73)--(-1.5, -0.866);
    \draw [fill=lightgray] (-1,-1.73)--(-0.32, -1.57)--(-0.52, -1.23)--(-1,-1.73);
    \draw [fill=lightgray] (-1,-1.73)--(-0.32, -1.57)--(0,-1.73)--(-1,-1.73);
    \draw [fill=lightgray] (-0.32, -1.57)--(0,-1.73)--(0.02, -1.37)--(-0.32, -1.57);
    \draw [fill=lightgray] (0,-1.73)--(0.02, -1.37)--(1,-1.73)--(0,-1.73);
    \draw [fill=lightgray] (0.02, -1.37)--(1,-1.73)--(0.5,-0.866)--(0.02,-1.37);
    \draw [fill=lightgray] (0.5,-0.866)--(0.02,-1.37)--(-0.18,-1.03)--(0.5,-0.866);
    \draw [fill=lightgray] (-0.18,-1.03)--(0.5,-0.866)--(0,0)--(-0.18,-1.03);
    \draw [fill=lightgray] (0,0)--(-0.18,-1.03)--(-0.5,-0.866)--(0,0);
    \draw [fill=lightgray] (-0.18,-1.03)--(-0.5,-0.866)--(-0.52, -1.23)--(-0.18,-1.03);
    \draw [fill=lightgray] (-0.5,-0.866)--(-0.52, -1.23)--(-1,-1.73)--(-0.5, -0.866);
    \draw [fill=lightgray] (1,-1.73)--(1.2,-1.06)--(0.8,-1.06)--(1,-1.73);
    \draw [fill=lightgray] (1,-1.73)--(1.2,-1.06)--(1.5,-0.866)--(1,-1.73);
    \draw [fill=lightgray] (1.2,-1.06)--(1.5,-0.866)--(1.2, -0.66)--(1.2,-1.06);
    \draw [fill=lightgray] (1.5,-0.866)--(1.2, -0.66)--(2,0)--(1.5,-0.866);
    \draw [fill=lightgray] (1.2, -0.66)--(2,0)--(1,0)--(1.2,-0.66);
    \draw [fill=lightgray] (1,0)--(1.2,-0.66)--(0.8,-0.66)--(1,0);
    \draw [fill=lightgray] (0.8,-0.66)--(1,0)--(0,0)--(0.8,-0.66);
    \draw [fill=lightgray] (0,0)--(0.8,-0.66)--(0.5,-0.866)--(0,0);
    \draw [fill=lightgray] (0.8,-0.66)--(0.5,-0.866)--(0.8,-1.06)--(0.8,-0.66);
    \draw [fill=lightgray] (0.5,-0.866)--(0.8,-1.06)--(1,-1.73)--(0.5, -0.866);

    \draw [fill=lightgray] (6,0.2)--(6.2,0.67)--(5.8,0.67)--(6,0.2);
    \draw [fill=lightgray] (6,0.2)--(6.2,0.67)--(6.5,0.866)--(6,0.2);
    \draw [fill=lightgray] (6.2,0.67)--(6.5,0.866)--(6.2, 1.07)--(6.2,0.67);
    \draw [fill=lightgray] (6.5,0.866)--(6.2, 1.07)--(7,1.73)--(6.5,0.866);
    \draw [fill=lightgray] (6.2, 1.07)--(7,1.73)--(6,1.73)--(6.2,1.07);
    \draw [fill=lightgray] (6,1.73)--(6.2,1.07)--(5.8,1.07)--(6,1.73);
    \draw [fill=lightgray] (5.8,1.07)--(6,1.73)--(5,1.73)--(5.8,1.07);
    \draw [fill=lightgray] (5,1.73)--(5.8,1.07)--(5.5,0.866)--(5,1.73);
    \draw [fill=lightgray] (5.8,1.07)--(5.5,0.866)--(5.8,0.67)--(5.8,1.07);
    \draw [fill=lightgray] (5.5,0.866)--(5.8,0.67)--(6,0.2)--(5.5, 0.866);
    \draw [fill=lightgray] (6,0.2)--(6.5,0.866)--(6,-0.2)--(6,0.2);
    \draw [fill=lightgray] (6,-0.2)--(6.5, 0.866)--(6.48, 0.5)--(6,-0.2);
    \draw [fill=lightgray] (6,-0.2)--(6.48,0.5)--(6.68,0.16)--(6,-0.2);
    \draw [fill=lightgray] (6,-0.2)--(6.68, 0.16)--(7,0)--(6,-0.2);
    \draw [fill=lightgray] (6.68, 0.16)--(7,0)--(7.02, 0.36)--(6.68, 0.16);
    \draw [fill=lightgray] (7,0)--(7.02, 0.36)--(8,0)--(7,0);
    \draw [fill=lightgray] (7.02, 0.36)--(8,0)--(7.5,0.866)--(7.02,0.36);
    \draw [fill=lightgray] (7.5,0.866)--(7.02,0.36)--(6.82,0.7)--(7.5,0.866);
    \draw [fill=lightgray] (6.82,0.7)--(7.5,0.866)--(7,1.73)--(6.82,0.7);
    \draw [fill=lightgray] (7,1.73)--(6.82,0.7)--(6.5,0.866)--(7,1.73);
    \draw [fill=lightgray] (6.82,0.7)--(6.5,0.866)--(6.48, 0.5)--(6.82,0.7);
    \draw [fill=lightgray] (4,0)--(4.68, 0.16)--(4.48, 0.5)--(4,0);
    \draw [fill=lightgray] (4,0)--(4.68, 0.16)--(5,0)--(4,0);
    \draw [fill=lightgray] (4.68, 0.16)--(5,0)--(5.02, 0.36)--(4.68, 0.16);
    \draw [fill=lightgray] (5,0)--(5.02, 0.36)--(6,0.2)--(5,0);
    \draw [fill=lightgray] (5.02, 0.36)--(6,0.2)--(5.5,0.866)--(5.02,0.36);
    \draw [fill=lightgray] (5.5,0.866)--(5.02,0.36)--(4.82,0.7)--(5.5,0.866);
    \draw [fill=lightgray] (4.82,0.7)--(5.5,0.866)--(5,1.73)--(4.82,0.7);
    \draw [fill=lightgray] (5,1.73)--(4.82,0.7)--(4.5,0.866)--(5,1.73);
    \draw [fill=lightgray] (4.82,0.7)--(4.5,0.866)--(4.48, 0.5)--(4.82,0.7);
    \draw [fill=lightgray] (4.5,0.866)--(4.48, 0.5)--(4,0)--(4.5, 0.866);
    \draw [fill=lightgray] (5,-1.73)--(5.2,-1.06)--(4.8,-1.06)--(5,-1.73);
    \draw [fill=lightgray] (5,-1.73)--(5.2,-1.06)--(5.5,-0.866)--(5,-1.73);
    \draw [fill=lightgray] (5.2,-1.06)--(5.5,-0.866)--(5.2, -0.66)--(5.2,-1.06);
    \draw [fill=lightgray] (5.5,-0.866)--(5.2, -0.66)--(6,0.2)--(5.5,-0.866);
    \draw [fill=lightgray] (5.2, -0.66)--(6,0.2)--(5,0)--(5.2,-0.66);
    \draw [fill=lightgray] (5,0)--(5.2,-0.66)--(4.8,-0.66)--(5,0);
    \draw [fill=lightgray] (4.8,-0.66)--(5,0)--(4,0)--(4.8,-0.66);
    \draw [fill=lightgray] (4,0)--(4.8,-0.66)--(4.5,-0.866)--(4,0);
    \draw [fill=lightgray] (4.8,-0.66)--(4.5,-0.866)--(4.8,-1.06)--(4.8,-0.66);
    \draw [fill=lightgray] (4.5,-0.866)--(4.8,-1.06)--(5,-1.73)--(4.5, -0.866);
    \draw [fill=lightgray] (5,-1.73)--(5.68, -1.57)--(5.48, -1.23)--(5,-1.73);
    \draw [fill=lightgray] (5,-1.73)--(5.68, -1.57)--(6,-1.73)--(5,-1.73);
    \draw [fill=lightgray] (5.68, -1.57)--(6,-1.73)--(6.02, -1.37)--(5.68, -1.57);
    \draw [fill=lightgray] (6,-1.73)--(6.02, -1.37)--(7,-1.73)--(6,-1.73);
    \draw [fill=lightgray] (6.02, -1.37)--(7,-1.73)--(6.5,-0.866)--(6.02,-1.37);
    \draw [fill=lightgray] (6.5,-0.866)--(6.02,-1.37)--(5.82,-1.03)--(6.5,-0.866);
    \draw [fill=lightgray] (5.82,-1.03)--(6.5,-0.866)--(6,-0.2)--(5.82,-1.03);
    \draw [fill=lightgray] (6,-0.2)--(5.82,-1.03)--(5.5,-0.866)--(6,-0.2);
    \draw [fill=lightgray] (5.82,-1.03)--(5.5,-0.866)--(5.48, -1.23)--(5.82,-1.03);
    \draw [fill=lightgray] (5.5,-0.866)--(5.48, -1.23)--(5,-1.73)--(5.5, -0.866);
    \draw [fill=lightgray] (5.5,-0.866)--(6,-0.2)--(6,0.2)--(5.5,-0.866);
    \draw [fill=lightgray] (7,-1.73)--(7.2,-1.06)--(6.8,-1.06)--(7,-1.73);
    \draw [fill=lightgray] (7,-1.73)--(7.2,-1.06)--(7.5,-0.866)--(7,-1.73);
    \draw [fill=lightgray] (7.2,-1.06)--(7.5,-0.866)--(7.2, -0.66)--(7.2,-1.06);
    \draw [fill=lightgray] (7.5,-0.866)--(7.2, -0.66)--(8,0)--(7.5,-0.866);
    \draw [fill=lightgray] (7.2, -0.66)--(8,0)--(7,0)--(7.2,-0.66);
    \draw [fill=lightgray] (7,0)--(7.2,-0.66)--(6.8,-0.66)--(7,0);
    \draw [fill=lightgray] (6.8,-0.66)--(7,0)--(6,-0.2)--(6.8,-0.66);
    \draw [fill=lightgray] (6,-0.2)--(6.8,-0.66)--(6.5,-0.866)--(6,-0.2);
    \draw [fill=lightgray] (6.8,-0.66)--(6.5,-0.866)--(6.8,-1.06)--(6.8,-0.66);
    \draw [fill=lightgray] (6.5,-0.866)--(6.8,-1.06)--(7,-1.73)--(6.5, -0.866);

    \fill (0,0) circle (1pt) node[below] {$\scriptstyle w_7$};
    \fill (1,0) circle (1pt); 
    \fill (2,0) circle (1pt) node[right] {$\scriptstyle w_8$};
    \fill (1,1.73) circle (1pt) node[above right] {$\scriptstyle w_{11}$};
    \fill (0.5, 0.866) circle (1pt); 
    \fill (-1,0) circle (1pt); 
    \fill (-0.5, 0.866) circle (1pt); 
    \fill (-2,0) circle (1pt) node[left] {$\scriptstyle w_6$};
    \fill (-1,1.73) circle (1pt) node[above left] {$\scriptstyle w_{10}$};
    \fill (1,-1.73) circle (1pt) node[below right] {$\scriptstyle w_{4}$};
    \fill (0.5, -0.866) circle (1pt); 
    \fill (-0.5,-0.866) circle (1pt); 
    \fill (-1,-1.73) circle (1pt) node[below left] {$\scriptstyle w_{3}$};
    \fill (0, 1.73) circle (1pt) node[above] {$\scriptstyle w'_{10,11}$};
    \fill (1.5, 0.866) circle (1pt) node[right] {$\scriptstyle w'_{8,11}$};
    \fill (-1.5, 0.866) circle (1pt) node[left] {$\scriptstyle w'_{6,10}$};
    \fill (0, -1.73) circle (1pt) node[below] {$\scriptstyle w'_{3,4}$};
    \fill (1.5, -0.866) circle (1pt) node[right] {$\scriptstyle w'_{4,8}$};
    \fill (-1.5, -0.866) circle (1pt) node[left] {$\scriptstyle w'_{3,6}$};
    \fill (-0.2,1.07) circle (0.7pt);
    \fill (-0.2,0.67) circle (0.7pt);
    \fill (0.2,1.07) circle (0.7pt);
    \fill (0.2,0.67) circle (0.7pt);
    
    \fill (0.48,0.5) circle (0.7pt);
    \fill (0.82,0.7) circle (0.7pt);
    \fill (1.02,0.36) circle (0.7pt);
    \fill (0.68,0.16) circle (0.7pt);
    
    \fill (-1.52,0.5) circle (0.7pt);
    \fill (-1.18,0.7) circle (0.7pt);
    \fill (-0.98,0.36) circle (0.7pt);
    \fill (-1.32,0.16) circle (0.7pt);
    
    \fill (-1.2,-0.66) circle (0.7pt);
    \fill (-1.2,-1.06) circle (0.7pt);
    \fill (-0.8,-0.66) circle (0.7pt);
    \fill (-0.8,-1.06) circle (0.7pt);
    
    \fill (-0.52,-1.23) circle (0.7pt);
    \fill (-0.18,-1.03) circle (0.7pt);
    \fill (0.02,-1.37) circle (0.7pt);
    \fill (-0.32,-1.57) circle (0.7pt);
    
    \fill (0.8,-0.66) circle (0.7pt);
    \fill (0.8,-1.06) circle (0.7pt);
    \fill (1.2,-0.66) circle (0.7pt);
    \fill (1.2,-1.06) circle (0.7pt);
    
    \fill (6,0.2) circle (1pt) node[above left] {$\scriptstyle w_{7_1}$};
    \fill (6,-0.2) circle (1pt) node[right] {$\scriptstyle w_{7_2}$};
    \fill (7,0) circle (1pt); 
    \fill (8,0) circle (1pt) node[right] {$\scriptstyle w_8$};
    \fill (7,1.73) circle (1pt) node[above right] {$\scriptstyle w_{11}$};
    \fill (6.5, 0.866) circle (1pt); 
    \fill (5,0) circle (1pt); 
    \fill (5.5, 0.866) circle (1pt); 
    \fill (4,0) circle (1pt) node[left] {$\scriptstyle w_6$};
    \fill (5,1.73) circle (1pt) node[above left] {$\scriptstyle w_{10}$};
    \fill (7,-1.73) circle (1pt) node[below right] {$\scriptstyle w_{4}$};
    \fill (6.5, -0.866) circle (1pt); 
    \fill (5.5,-0.866) circle (1pt); 
    \fill (5,-1.73) circle (1pt) node[below left] {$\scriptstyle w_{3}$};
    \fill (6, 1.73) circle (1pt) node[above] {$\scriptstyle w'_{10,11}$};
    \fill (7.5, 0.866) circle (1pt) node[right] {$\scriptstyle w'_{8,11}$};
    \fill (4.5, 0.866) circle (1pt) node[left] {$\scriptstyle w'_{6,10}$};
    \fill (6, -1.73) circle (1pt) node[below] {$\scriptstyle w'_{3,4}$};
    \fill (7.5, -0.866) circle (1pt) node[right] {$\scriptstyle w'_{4,8}$};
    \fill (4.5, -0.866) circle (1pt) node[left] {$\scriptstyle w'_{3,6}$};
    \fill (5.8,1.07) circle (0.7pt);
    \fill (5.8,0.67) circle (0.7pt);
    \fill (6.2,1.07) circle (0.7pt);
    \fill (6.2,0.67) circle (0.7pt);
    
    \fill (6.48,0.5) circle (0.7pt);
    \fill (6.82,0.7) circle (0.7pt);
    \fill (7.02,0.36) circle (0.7pt);
    \fill (6.68,0.16) circle (0.7pt);
    
    \fill (4.48,0.5) circle (0.7pt);
    \fill (4.82,0.7) circle (0.7pt);
    \fill (5.02,0.36) circle (0.7pt);
    \fill (4.68,0.16) circle (0.7pt);
    
    \fill (4.8,-0.66) circle (0.7pt);
    \fill (4.8,-1.06) circle (0.7pt);
    \fill (5.2,-0.66) circle (0.7pt);
    \fill (5.2,-1.06) circle (0.7pt);
    
    \fill (5.48,-1.23) circle (0.7pt);
    \fill (5.82,-1.03) circle (0.7pt);
    \fill (6.02,-1.37) circle (0.7pt);
    \fill (5.68,-1.57) circle (0.7pt);
    
    \fill (6.8,-0.66) circle (0.7pt);
    \fill (6.8,-1.06) circle (0.7pt);
    \fill (7.2,-0.66) circle (0.7pt);
    \fill (7.2,-1.06) circle (0.7pt);
    
    \draw[-stealth] (2.5,0)--(3.5,0);
    
\end{tikzpicture}
\end{center}
\caption{Replacing a degree 14 vertex in $\td{T_4}$ when $k=20$.}
\label{fig:Replacing_vertex}
\end{figure}

After repeating the above process for each degree 14 vertex in $\widetilde{T_i}$, we take the clique complex and call the resulting flag complex $Y_2$. Observe that this process increases the number of vertices by 1, the number of edges by 3, and the number of faces by 2 each time a degree 14 vertex in $\widetilde{T_i}$ is replaced. Also, note that $\maxdeg(Y_2)\leq 12$ for all $m$.

Now, we give the $w_j$, $w'_{s,t},$ and $u_\ell$ vertices their natural orderings and say that $w'_{s,t} > w_j$ and $w'_{s,t}>u_\ell$ for all $\ell, s,t,$ and $j$, and then let these vertex orderings induce orientations on the edges and faces of $Y_2$ (as shown in Figure \ref{fig:triangulations}). Counting the vertices, edges, and faces of $Y_2$ we have that if $0\leq k\leq 12$, then there were no degree 14 vertices to remove, so $|V(Y_2)|=6k+2\delta +2$, $|E(Y_2)|=17k+6\delta$, and $|F(Y_2)|=10k+4\delta$. If $k\geq 13$, then $i\geq 3$ and at least one degree 14 vertex was removed to construct $Y_2$ from $\widetilde{T_i}$.  Table \ref{tab:Y2counts} gives the number of vertices, edges, and faces of $Y_2$ for all values of $k\geq 13$.

\begin{table}[ht]
\begin{center}
    \def\arraystretch{1.5}
    \setlength\tabcolsep{10pt}
    \begin{tabular}{|c|c|c|c|c|}
    \hline
    $k$ & $\delta$ & $|V(Y_2)|$ & $|E(Y_2)|$ & $|F(Y_2)|$ \\
    \hline
    $4i+4$ & 0 & $\frac{13}{2}k-4$ & $\frac{37}{2}k-18$ & $11k-12$ \\
    \hline
    $4i+3$ & 1 & $\frac{13}{2}k-\frac{3}{2}$ & $\frac{37}{2}k-\frac{21}{2}$ & $11k-7$ \\
    \hline
    $4i+2$ & 2 & $\frac{13}{2}k$ & $\frac{37}{2}k-6$ & $11k-4$ \\
    \hline
    $4i+1$ & 3 & $\frac{13}{2}k+\frac{5}{2}$ & $\frac{37}{2}k+\frac{3}{2}$ & $11k+1$ \\
    \hline
    \end{tabular}
\end{center}
\caption{Number of vertices, edges, and faces in $Y_2$ when $k\geq 13$.}
\label{tab:Y2counts}
\end{table}

\subsubsection{Homology of $Y_2$}  Since $Y_2$ is an oriented flag triangulation of $S^2$ with $k$ square holes, each of which are vertex disjoint and nonadjacent, our $Y_2$ is homeomorphic to Newman's $Y_2$ in  the $d=2$ case of \cite[Lemma 5.7]{newman}, and we can apply the same argument
to compute the homology of $Y_2$. 
We denote the 1-cycles that are the boundaries of the $k$ square holes by $\tau_1,\dots, \tau_k$. Explicitly, for $j=1, \dots, k$, we define \[\tau_{j}: = [u_{4j-4}, u_{4j-3}] + [u_{4j-3}, u_{4j-2}] + [u_{4j-2}, u_{4j-1}] - [u_{4j-4}, u_{4j-1}].\] Then, by our construction, each $\tau_j$ is a positively-oriented 1-cycle in $H_1(Y_2)$, and exactly as in \cite[Proof of Lemma 5.7]{newman}, we have that
$
    H_1(Y_2) = \langle \tau_1, \ldots, \tau_k \vert \tau_1 + \cdots + \tau_k = 0 \rangle.
$
%
\subsection{Construction of $X$ and proof of Theorem~\ref{thm:Xm}}
Now we attach $Y_1$ and $Y_2$ together to form the two-dimensional flag complex $X$ such that the torsion subgroup of $H_1(X)$ is isomorphic to $\mathbb{Z}/m\mathbb{Z}$. This part essentially follows~\cite[\S3]{newman}, though we must confirm that the resulting complex is flag and satisfies the desired bound of vertex degree.

\begin{proof}[Proof of Theorem~\ref{thm:Xm}]
For a given $m$, let $Y_1$ and $Y_2$ be the complexes constructed in the previous subsections.  Let $S$ denote the subcomplex of $Y_2$ induced by the $4k$ vertices $u_0,\dots, u_{4k-1}$. Since the square holes in $Y_2$ are vertex-disjoint and have no edges between any two of them, $S$ is a disjoint union of $k$ square boundaries. Let $f: S \rightarrow Y_1$ be the simplicial map defined, for $j=1, \ldots, k$, by
\begin{align*}
     & u_{4j-4} \mapsto v_{4n_j}, & & u_{4j-3} \mapsto v_{4n_j + 1}, & u_{4j-2} \mapsto v_{4n_j + 2}, && u_{4j-1} \mapsto v_{4n_j + 3}.
\end{align*}

Following \cite[\S3]{newman}, let $X = Y_1 \sqcup_f Y_2$ and observe that this is a simplicial complex by the same argument as Newman gives. In addition, $X$ is a flag complex because $Y_1$ and $Y_2$ are flag, and we subdivided the edges of $Y_1$ and $Y_2$ to avoid the possibility that $X$ might contain a 3-cycle which doesn't have a face. Furthermore, in $X$ the squares $\tau_j$ and $\gamma_{n_j}$ are identified by $f$ for $j=1, \ldots, k$, and, as in \cite{newman},
\begin{equation*}
    H_1(X) \cong \mathbb{Z}^{k-1} \oplus \mathbb{Z}/m \mathbb{Z},
\end{equation*}
where $\Z/m\Z$ has the repeated squares representation given by
$$\la \gamma_0,\gamma_1,\dots,\gamma_{n_k}\; |\; 2\gamma_0=\gamma_1, 2\gamma_1=\gamma_2,\dots, 2\gamma_{n_k-1}=\gamma_{n_k}, \gamma_{n_1}+\cdots+\gamma_{n_k}=0 \ra.$$
Finally, using our counts for the number of vertices, edges, and faces of $Y_1$ and $Y_2$ and with $\delta$ defined as above,
if $0\leq k\leq 12$, we have
$$|V(X)|=2k+12n_k+6+2\delta, \; |E(X)|= 13k+40n_k+4+6\delta, \text{ and }|F(X)|=10k+28n_k+4\delta.$$ If $k\geq 13$, then Table \ref{tab:Xcounts} gives the number of vertices, edges, and faces in $X$ (where $i=\lfloor \frac{k-1}{4} \rfloor$). 

\begin{table}[ht]
\begin{center}
    \def\arraystretch{1.5}
    \setlength\tabcolsep{10pt}
    \begin{tabular}{|c|c|c|c|c|}
    \hline
    $k$ & $\delta$ & $|V(X)|$ & $|E(X)|$ & $|F(X)|$ \\
    \hline
    $4i+4$ & 0 & $\frac{5}{2}k+12n_k$ & $\frac{29}{2}k+40n_k-14$ & $11k+28n_k-12$ \\
    \hline
    $4i+3$ & 1 & $\frac{5}{2}k+12n_k+\frac{5}{2}$ & $\frac{29}{2}k+40n_k-\frac{13}{2}$ & $11k+28n_k-7$ \\
    \hline
    $4i+2$ & 2 & $\frac{5}{2}k+12n_k+4$ & $\frac{29}{2}k+40n_k-2$ & $11k+28n_k-4$ \\
    \hline
    $4i+1$ & 3 & $\frac{5}{2}k+12n_k+\frac{13}{2}$ & $\frac{29}{2}k+40n_k+\frac{11}{2}$ & $11k+28n_k+1$ \\
    \hline
    \end{tabular}
\end{center}
\caption{Number of vertices, edges, and faces in $X$ when $k\geq 13$.}
\label{tab:Xcounts}
\end{table}
Additionally, recall that $\maxdeg(Y_1)\leq9$ and $\maxdeg(Y_2)\leq 12$. Since in $X$ we are only identifying the squares of $Y_2$ with $k$ of the squares of $Y_1$, to find the maximum degree of any vertex of $X$, we need only check the degrees of the identified vertices. In $Y_1$, we know that $\deg(v_j)\leq 9$ for each $j$, and in $Y_2$, we know that $\deg(u_\ell) \in \{4,5\}$ for each $\ell$. Let $v_j$ and $u_\ell$ be vertices that are identified in $X$. Since two of their adjacent edges in the squares are identified as well, in $X$ we see that $\deg(v_j)=\deg(u_\ell) \leq 12$. Thus, $\maxdeg(X)\leq 12$.
\end{proof}

We also note the following corollary:
\begin{corollary}
\label{cor:flag-newman}
    For every finite abelian group $G$ there is a two-dimensional flag complex $X$ such that the torsion subgroup of $H_1(X)$ is isomorphic to $G$ and $\maxdeg(X)\leq 12$.
\end{corollary}
\begin{proof}
Let $G = \Z/m_1\Z \oplus \Z/m_2\Z \oplus \cdots \oplus \Z/m_r \Z$ with $m_1|m_2|\cdots|m_r$ be an arbitrary finite abelian group. By Theorem \ref{thm:Xm}, there exist two-dimensional flag complexes $X_{m_i}$ such that the torsion subgroup of $H_1(X_{m_i})$ is isomorphic to $\Z/m_i\Z$ and $\maxdeg(X_{m_i})\leq 12$. If $X$ is the disjoint union of all the $X_{m_i}$, then $X$ satisfies the hypotheses of the corollary.
\end{proof}

\section{Appearance of subcomplexes in $\Delta(n,p)$}\label{sec:subgraphs}
The goal of this section is to show that, for attaching probabilities $p$ in an appropriate range, the flag complex $X_m$ from Theorem~\ref{thm:Xm} will appear with high probability as an induced subcomplex of $\Delta(n,p)$.  See \S\ref{sec:background} for the relevant definitions and notation used throughout this section.  Here is our main result:

\begin{proposition}\label{prop:high-probability}
Let $m\geq 2$, and let $X_m$ be as in Theorem~\ref{thm:Xm}. If $\Delta\sim \Delta(n,p)$ is a random flag complex with $n^{-1/6}\ll p\leq 1-\epsilon$ for some $\epsilon>0$, then $\P\left[X_m \overset{ind}{\subset} \Delta(n,p)\right]\rightarrow 1$ as $n \to \infty$.
\end{proposition}
Our proof of this result will rely on Bollob\'{a}s's theorem on the appearance of subgraphs of a random graph, which we state here for reference.
\begin{theorem}[Bollob\'{a}s \cite{Bollobas-subgraph}]
  \label{thm:Bollob\'{a}s}
  Let $G'$ be a fixed graph, let $m(G')$ be the essential density of $G'$ defined in Definition~\ref{mG}, and let $G(n,p)$ be the Erd\H{o}s-R\'enyi random graph on $n$ vertices with attaching probability $p$. As $n \to \infty$, we have 
    \[\P\left[G'\subset G(n,p)\right]\rightarrow \begin{cases}
    0 & \text{if } p\ll n^{-1/m(G')}\\
    1 & \text{if } p\gg n^{-1/m(G')}
    \end{cases}.\]
\end{theorem}

Since any flag complex is determined by its underlying graph, we can almost apply this to prove Proposition~\ref{prop:high-probability}.  However, Proposition~\ref{prop:high-probability} (and our eventual application of it via Hochster's formula to Theorem~\ref{thm:m torsion}) requires $X_m$ to appear as an induced subcomplex, whereas Bollob\'{a}s's result is for not necessarily induced subgraphs.  The following proposition, which is likely known to experts, shows that so long as $p$ is bounded away from $1$, this distinction is immaterial in the limit.

\begin{proposition}
\label{prop:induced-bollobas}
Let $G'$ be a fixed graph, let $m(G')$ be the essential density of $G'$ defined in Definition~\ref{mG}, and let $G(n,p)$ be the Erd\H{o}s-R\'enyi random graph on $n$ vertices with attaching probability $p$. Suppose $p = p(n)\leq 1-\epsilon$ for some $\epsilon>0$. Then as $n \to \infty$, we have
        \[\P\left[G'\overset{ind}{\subset} G(n,p)\right]\rightarrow \begin{cases}
        0 & \text{if } p\ll n^{-1/m(G')}\\
        1 & \text{if } p\gg n^{-1/m(G')}
    \end{cases}.\]
\end{proposition}

\begin{proof}
Since an induced subgraph is a subgraph, if $\P[G'\subset G(n,p)]\rightarrow 0$, then\\ $\P\left[G'\overset{ind}{\subset} G(n,p)\right]\rightarrow 0$. Thus, the first half of the threshold is a direct consequence of Theorem~\ref{thm:Bollob\'{a}s}, and all that needs to be shown is the second half of the threshold. 

Suppose that $p\gg n^{-1/m(G')}$.  We will mirror the proof of Bollob\`{a}s's theorem from \cite[Theorem~5.3]{frieze-book} (originally due to \cite{ruc-vince}), which relies on the second moment method. Let $\Lambda(G',n)$ be the set containing all of the possible ways that $G'$ can appear as a induced subgraph of $G(n,p)$.
Thus, an element $H\in \Lambda(G',n)$ corresponds to a subset of the $n$ vertices and specified edges among those vertices such that the resulting graph is a copy of $G'$.  We want to count the number of times $G'$ appears as an induced subgraph of $G(n,p)$.  For each $H\in \Lambda(G',n)$, we let $\mathbf{1}_{H}$ be the corresponding indicator random variable, where $\mathbf{1}_H = 1$ occurs in the event that restricting $G(n,p)$ to the vertices of $H$ is precisely the copy of $G'$ indicated by $H$. 
Note that the random variables $\mathbf{1}_{H}$ are not independent, as two distinct elements from $\Lambda(G',n)$ might have overlapping vertex sets.
If we let $N_{G'}$ be the random variable for the number of copies of $G'$ appearing as induced subgraphs in $G(n,p)$, then we have $N_{G'} = \ds\sum_{H \in \Lambda(G',n)} \mathbf{1}_{H}.$

Our goal is to show that $\P[N_{G'}\geq 1]\to 1$, or equivalently that $\P[N_{G'}= 0]\to 0$. Since $N_{G'}$ is non-negative, the second moment method as seen in \cite[Theorem 4.3.1]{alon-spencer-book}
states that $\P[N_{G'}= 0]\leq \frac{\Var(N_{G'})}{\E[N_{G'}]^2}$, so it suffices to show that $\frac{\Var(N_{G'})}{\E[N_{G'}]^2}\rightarrow 0$. To start, we will bound the expected value.  To simplify notation throughout the following computation, we let $v=|V(G')|$ and $e=|E(G')|$ denote the number of vertices and edges of $G'$.
\begin{align*}
    \E[N_{G'}]&= \sum_{H\in \Lambda(G',n)} \E[\mathbf{1}_H]\\
    &= \sum_{H\in \Lambda(G',n)} p^{e}(1-p)^{\binom{v}{2}-e}\\
    &= \Omega(n^{v})\cdot p^{e}(1-p)^{\binom{v}{2}-e}.
\end{align*}

Now let us repeat this with the variance instead. 
\begin{align*}
    \Var(N_{G'}) &= \sum_{H,H'\in \Lambda(G',n)} \E[\mathbf{1}_{H}\mathbf{1}_{H'}] - \E[\mathbf{1}_{H}]\E[\mathbf{1}_{H'}]\\
    &= \sum_{H,H'\in \Lambda(G',n)} \P[\mathbf{1}_{H}=1\text{ and }\mathbf{1}_{H'}=1] - \P[\mathbf{1}_{H}=1]\P[\mathbf{1}_{H'}=1]\\
    &= \sum_{H,H'\in \Lambda(G',n)} \P[\mathbf{1}_{H}=1]\left(\P[\mathbf{1}_{H'}=1 \mid \mathbf{1}_{H}=1]- \P[\mathbf{1}_{H'}=1]\right)\\
    &= p^{e}(1-p)^{\binom{v}{2}-e} \sum_{H,H'\in \Lambda(G',n)}\P[\mathbf{1}_{H'}=1 \mid \mathbf{1}_{H}=1]- \P[\mathbf{1}_{H'}=1]\\
    \intertext{If $H$ and $H'$ don't share at least two vertices, $\mathbf{1}_H$ and $\mathbf{1}_{H'}$ are independent of each other.  So we can restrict to the case where they share at least two vertices, which gives}
    &= p^{e}(1-p)^{\binom{v}{2}-e} \sum_{i=2}^{v}\sum_{\substack{H,H'\in \Lambda(G',n) \\ |V(H)\cap V(H')|=i}}\P[\mathbf{1}_{H'}=1 \mid \mathbf{1}_{H}=1]- \P[\mathbf{1}_{H'}=1].
\end{align*}
We now come to the key observation, which is also at the heart of the proof in \cite[Theorem~5.3]{frieze-book}: $\P[\mathbf{1}_{H'}=1 \mid \mathbf{1}_{H}=1]$ is maximized if those edges and non-edges in $H$ are exactly those that are required by $H'$. Thus, by applying the fact that any subgraph of $G'$ with $i$ vertices, has at most $i\cdot m(G')$ edges and at most $\binom{i}{2}$ non-edges we get the following bound for $H,H'\in \Lambda(G',n)$ sharing $i$ vertices:
    \[\P[\mathbf{1}_{H'}=1 \mid \mathbf{1}_{H}=1]\leq \P[\mathbf{1}_{H'}=1]\cdot p^{-i\cdot m(G')}(1-p)^{-\binom{i}{2}}\]
From here, it is a standard computation.  Substituting this back into the previous equation and simplifying, we get
    \begin{align*} 
    \Var(N_{G'}) &\leq p^{e}(1-p)^{\binom{v}{2}-e} \sum_{i=2}^{v}\sum_{\substack{H,H'\in \Lambda(G',n) \\ |V(H)\cap V(H')|=i}}\P[\mathbf{1}_{H'}=1]\left(p^{-i\cdot m(G')}(1-p)^{-\binom{i}{2}}-1\right)\\
    &\leq \left(p^{e}(1-p)^{\binom{v}{2}-e}\right)^2 \sum_{i=2}^{v} O\left(n^{2v-i}\right)\left(p^{-i\cdot m(G')}(1-p)^{-\binom{i}{2}}-1\right).\\
    \intertext{And since $p$ is bounded away from $1$ and $1-p$ is bounded away from $0$, we get}
    &\leq \left(p^{e}(1-p)^{\binom{v}{2}-e}\right)^2 \sum_{i=2}^{v} O\left(n^{2v-i}p^{-i\cdot m(G')}\right).
    \end{align*}
Finally, applying the second moment method gives 
\[
\P[N_{G'}= 0]\leq \frac{\Var(N_{G'})}{\E[N_{G'}]^2}=\frac{\ds\sum_{i=2}^{v} O\left(n^{2v-i}p^{-i\cdot m(G')}\right)}{\Omega(n^{2v})} =\sum_{i=2}^{v} O\left(n^{-i}p^{-i\cdot m(G')}\right).
\]
Since $p\gg n^{-1/m(G')}$, we conclude that $np^{m(G')}\rightarrow \infty$, and therefore, $\P[N_{G'}=0 ]\rightarrow 0$. It follows that $\P\left[G'\overset{ind}{\subset} G(n,p)\right]\rightarrow 1$.
\end{proof}

We now turn to the proof of Proposition~\ref{prop:high-probability}.

\begin{proof}[Proof of Proposition~\ref{prop:high-probability}]
Recall that $X_m$ is the complex from Theorem~\ref{thm:Xm}, and let $H_m$ be its underlying graph. Moreover, the 
underlying graph of $\Delta(n,p)$ is the Erd\H{o}s-R\'enyi random graph $G(n,p)$.  Since a flag complex is uniquely determined by its underlying graph, it suffices to show that $\P\left[H_m\overset{ind}{\subset}G(n,p)\right]\rightarrow 1$.

Since $\maxdeg(H_m)\leq 12$, every subgraph has average degree at most $12$. Thus, the essential density $m(H_m)$ satisfies $m(H_m)\leq 6$. Since $p\gg n^{-1/6}$, we have $p\gg n^{-1/m(H_m)}$. Applying Proposition~\ref{prop:induced-bollobas} gives $\P\left[H_m \overset{ind}{\subset} G(n,p)\right]\rightarrow 1$; thus, $\P\left[X_m\overset{ind}{\subset}\Delta(n,p)\right]\rightarrow 1$. 
\end{proof}

\begin{remark}\label{rmk:sharpness}
Explicitly computing the essential density $m(H_m)$ seems difficult in general, and our chosen bound $m(H_m)\leq 6$, which is determined by the fact that $6 = \frac{1}{2}\maxdeg(X_m)$, is likely too coarse.
It would be interesting to see a sharper result on $m(H_m)$, as this could potentially provide a heuristic for decreasing the bound on $r$ in Conjecture~\ref{conj:dependence}.  Might it even be the case that $m(H_m)$ is half the average degree, $\frac{1}{2}\avg(H_m)$?

In any case, $\frac{1}{2}\avg(H_m)$ at least provides a lower bound on $m(H_m)$.  Due to the detailed nature of the constructions in \S\ref{sec:construction}, we can estimate this value.  Let $k\geq 13$ and $m\gg 0$ so that $n_k=\lfloor \log_2(m)\rfloor$ will be much larger than $\delta$.  By Table~\ref{tab:Xcounts}, the number of vertices will be approximately $\frac{5}{2}k+12n_k$ and the number of edges will be approximately $\frac{29}{2}k+40n_k$.  The smallest the ratio of edges to vertices can be is when $n_k\gg k$, in which case the ratio will be approximately $3\frac{1}{3}$.  A similar computation holds for $k\leq 12$ and for $m\gg 0$.  We can conclude that $m(H_m)\geq 3 \frac{1}{3}-\epsilon$, where $\epsilon$ is a positive constant that goes to $0$ as $m\to \infty$. \qed
\end{remark}

\section{A detailed analysis of 2-torsion}\label{sec:2torsion}
The goal of this section is to provide a more detailed analysis of what happens in the case of 2-torsion (when $m=2$ in Proposition \ref{prop:high-probability}). In \cite{costa-farber-horak}, Costa, Farber, and Horak analyze the $2$-torsion of the fundamental group of $\Delta(n,p)$. Their results, specifically Theorem 7.2, give that if $n^{-11/30} \ll p \ll n^{-1/3-\epsilon}$ where $0 < \epsilon < \frac{1}{30}$ is fixed, then $H_1(\Delta(n,p))$ has $2$-torsion with high probability as $n \to \infty$. Since our aim is to show that there is $2$-torsion with high probability in the homology of an induced subcomplex of $\Delta(n,p)$, rather than in the global homology, we are able to extend their threshold to $n^{-11/30} \ll p \leq 1-\epsilon$ where $\epsilon >0$. We use the same techniques as in \S\ref{sec:subgraphs}, but instead of using $X_2$ from Theorem \ref{thm:Xm}, we use a known flag triangulation of $\R P^2$ that minimizes the number of vertices and where we can easily compute its essential density. This gives the less restrictive threshold of $p\gg n^{-11/30}$ in the $2$-torsion case as opposed to $p\gg n^{-1/6}$ in the general case. In \cite[Figure 1]{bibby2019minimal}, the authors found two (nonisomorphic) minimal flag triangulations of $\R P^2$, each of which have 11 vertices and 30 edges and differ by a single bistellar 0-move; one of these is used in \cite{costa-farber-horak}, and the other, which we use in this section, is depicted in Figure \ref{fig:flagRP2}.

\begin{figure}
\begin{center}
\begin{tikzpicture}[scale=1.5]
    \draw [fill=lightgray, thick] (0,2)--(0.85,1.3)--(0,1)--(0,2);
    \draw [fill=lightgray, thick] (0,2)--(-0.85,1.3)--(0,1)--(0,2);
    \draw [fill=lightgray, thick] (0,1)--(0.85,1.3)--(0.8,0.3)--(0,1);
    \draw [fill=lightgray, thick] (0,1)--(-0.85,1.3)--(-0.8,0.3)--(0,1);
    \draw [fill=lightgray, thick] (0,1)--(0.8,0.3)--(0,0)--(0,1);
    \draw [fill=lightgray, thick] (0,1)--(-0.8,0.3)--(0,0)--(0,1);
    \draw [fill=lightgray, thick] (0.85,1.3)--(1.7,0.6)--(0.8,0.3)--(0.85,1.3);
    \draw [fill=lightgray, thick] (0.85,1.3)--(1.7,0.6)--(0.8,0.3)--(0.85,1.3);
    \draw [fill=lightgray, thick] (-0.85,1.3)--(-1.7,0.6)--(-0.8,0.3)--(-0.85,1.3);
    \draw [fill=lightgray, thick] (1.7,0.6)--(0.8,0.3)--(1.4,-0.5)--(1.7,0.6);
    \draw [fill=lightgray, thick] (-1.7,0.6)--(-0.8,0.3)--(-1.4,-0.5)--(-1.7,0.6);
    \draw [fill=lightgray, thick] (0.8,0.3)--(1.4,-0.5)--(0.5,-0.75)--(0.8,0.3);
    \draw [fill=lightgray, thick] (0.8,0.3)--(0,0)--(0.5,-0.75)--(0.8,0.3);
    \draw [fill=lightgray, thick] (-0.8,0.3)--(-1.4,-0.5)--(-0.5,-0.75)--(-0.8,0.3);
    \draw [fill=lightgray, thick] (-0.8,0.3)--(0,0)--(-0.5,-0.75)--(-0.8,0.3);
    \draw [fill=lightgray, thick] (0,0)--(-0.5,-0.75)--(0.5,-0.75)--(0,0);
    \draw [fill=lightgray, thick] (1.4,-0.5)--(1,-1.6)--(0.5,-0.75)--(1.4,-0.5);
    \draw [fill=lightgray, thick] (-1.4,-0.5)--(-1,-1.6)--(-0.5,-0.75)--(-1.4,-0.5);
    \draw [fill=lightgray, thick] (1,-1.6)--(0.5,-0.75)--(0,-1.6)--(1,-1.6);
    \draw [fill=lightgray, thick] (-1,-1.6)--(-0.5,-0.75)--(0,-1.6)--(-1,-1.6);
    \draw [fill=lightgray, thick] (0,-1.6)--(0.5,-0.75)--(-0.5,-0.75)--(0,-1.6);

    \fill (0,0) circle (2pt) node[above right] {$v_8$};
    \fill (0,1) circle (2pt) node[above right] {$v_9$};
    \fill (0,2) circle (2pt) node[above] {$v_3$};
    \fill (0.85,1.3) circle (2pt) node[above right] {$v_2$};
    \fill (1.7,0.6) circle (2pt) node[right] {$v_6$};
    \fill (-0.85,1.3) circle (2pt) node[above left] {$v_{10}$};
    \fill (-1.7,0.6) circle (2pt) node[left] {$v_5$};
    \fill (-1,-1.6) circle (2pt) node[below left] {$v_2$};
    \fill (-0.5,-0.75) circle (2pt) node[below left] {$v_7$};
    \fill (0.5,-0.75) circle (2pt) node[below right] {$v_4$};
    \fill (1,-1.6) circle (2pt) node[below right] {$v_{10}$};
    \fill (-1.4,-0.5) circle (2pt) node[below left] {$v_6$};
    \fill (1.4,-0.5) circle (2pt) node[below right] {$v_5$};
    \fill (0.8,0.3) circle (2pt) node[above right] {$v_1$};
    \fill (-0.8,0.3) circle (2pt) node[left] {$v_{11}$};
    \fill (0,-1.6) circle (2pt) node[below] {$v_3$};
\end{tikzpicture}
\end{center}
\caption{A minimal flag triangulation of $\R P^2$, denoted by $\Delta(G)$.}
\label{fig:flagRP2}
\end{figure}
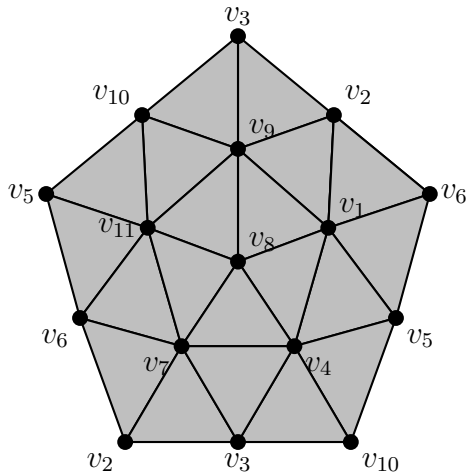

For the remainder of this section, let $G$ denote the underlying graph of this flag triangulation of $\R P^2$, which we denote by $\Delta(G)$.  To understand the probability that $\Delta(G)$ appears as an induced subcomplex of $\Delta(n,p)$, we need to compute the essential density $m(G)$.  
\begin{lemma}\label{lem:bollobas of G}
For the graph $G$ underlying the flag triangulation of $\R P^2$ exhibited in Figure~\ref{fig:flagRP2}, the essential density $m(G)$ is $30/11$.
\end{lemma}
\begin{proof}
This amounts to an exhaustive computation, which is summarized in Table~\ref{tab:RP2 counts}.  In particular, Table~\ref{tab:RP2 counts} identifies the maximal number of edges that a subgraph $H\subset G$ on $|V(H)|$ vertices can have, for each $|V(H)|\leq 11$.  One can see from the table that $m(G)$ is maximized by the entire graph, and thus $m(G) = |E(G)|/|V(G)| = 30/11$. 
\end{proof}

Lemma~\ref{lem:bollobas of G} shows that the graph $G$ is strongly balanced in the sense of Definition~\ref{mG}.  While we expect the essential density of our complexes $X_m$ to be lower than the coarse bound of $\frac{1}{2}\maxdeg(X_m)$ (see Remark~\ref{rmk:sharpness}), we note that in the case of the graph $G$, this difference is not very large.  In fact, we have $\frac{1}{2}\maxdeg(G)=3$ and $m(G)=30/11\approx 2.72$.
\begin{table}
\begin{center}
\def\arraystretch{1.5}
    \begin{tabular}{|c|c|c|c|}
    \hline
    $|V(H)|$ & $\max\{|E(H)|\}$ & $V(H)$ & $\max\left\{\frac{|E(H)|}{|V(H)|}\right\}$\\
    \hline
    1 & 0 & $\{v_1\}$ & 0\\
    \hline
    2 & 1 & $\{v_1, v_2\}$ & $\frac12$\\
    \hline
    3 & 3 & $\{v_1, v_2, v_6\}$ & 1\\
    \hline
    4 & 5 & $\{v_1,v_2,v_5,v_6\}$ & $\frac54$\\
    \hline
    5 & 7 & $\{v_1,v_2,v_4,v_5,v_6\}$ & $\frac75$\\
    \hline
    6 & 10 & $\{v_1,v_4,v_7,v_8,v_9,v_{11}\}$ &  $\frac53$\\
    \hline
    7 & 13 & $\{v_1,v_2,v_4,v_7,v_8,v_9,v_{11}\}$ & $\frac{13}{7}$\\
    \hline
    8 & 17 & $\{v_1,v_2,v_4,v_6,v_7,v_8,v_9,v_{11}\}$ & $\frac{17}{8}$\\
    \hline
    9 & 21 & $\{v_1,v_2,v_3,v_4,v_6,v_7,v_8,v_9,v_{11}\}$ & $\frac{7}{3}$\\
    \hline
    10 & 25 & $\{v_1,v_2,v_3,v_4,v_5,v_6,v_7,v_8,v_9,v_{11}\}$ & $\frac{5}{2}$\\
    \hline
    11 & 30 & $\{v_1,\dots,v_{11}\}$ & $\frac{30}{11}$\\
    \hline
    \end{tabular}
\end{center}
\caption{With $G$ as the underlying graph of the complex in Figure~\ref{fig:flagRP2}, this table computes the maximal number of edges of subgraphs $H\subset G$ with varying number of vertices.}
\label{tab:RP2 counts}
\end{table}
Combining Lemma~\ref{lem:bollobas of G} and Theorem \ref{thm:Bollob\'as} we obtain an analogue of Proposition~\ref{prop:high-probability}.

\begin{proposition}\label{prop:high-prob 2tor}
If $\Delta\sim \Delta(n,p)$ is a random flag complex with $n^{-11/30}\ll p\leq 1-\epsilon$ for some $\epsilon>0$,  then  $\P\left[\Delta(G) \overset{ind}{\subset} \Delta(n,p)\right]\rightarrow 1$ as $n \to \infty$.
\end{proposition}
\begin{proof}
The proof is nearly identical to that of Proposition~\ref{prop:high-probability}, so we omit the details.  
\end{proof}

\begin{question}\label{q:2torsion threshold}
It would be interesting to know whether $p \gg n^{-11/30}$ is a sharp threshold for the appearance of 2-torsion in the homology of any induced subcomplex of $\Delta(n,p)$. While ~\cite[Theorem~7.1]{costa-farber-horak} shows that the global homology has no torsion if $p\ll n^{-11/30}$, it is possible that some induced subcomplex of $\Delta(n,p)$ has $2$-torsion. A closely related question is whether there exists a flag complex $X$ with 2-torsion homology and a smaller essential density than $30/11$.
\end{question}

\section{Torsion in the Betti tables associated to $\Delta$}\label{sec:Betti numbers}

We now prove Theorem~\ref{thm:m torsion}.  The hard work was done in the previous sections.
\begin{proof}[Proof of Theorem~\ref{thm:m torsion} (2)]
Assume $n^{-1/6}\ll p \leq 1-\epsilon$ and let $\Delta\sim \Delta(n,p)$.  Let $X_m$ be as constructed in the proof of Theorem~\ref{thm:Xm}.  By Proposition~\ref{prop:high-probability}, $\Delta$ contains $X_m$ as an induced subcomplex with high probability as $n \to \infty$. Since $H_1(X_m)$ has $m$-torsion, Hochster's formula (see Fact~\ref{fact:depend}) gives that the Betti table of the Stanley--Reisner ideal of $\Delta$ has $\ell$-torsion for every prime $\ell$ dividing $m$.
\end{proof}

We can also apply the more detailed study of $2$-torsion from \S\ref{sec:2torsion} to obtain a result on the appearance of $2$-torsion in the Betti tables of random flag complexes.  
\begin{proposition}\label{prop:2tors}
Let $\Delta\sim \Delta(n,p)$ be a random flag complex with $n^{-11/30}\ll p \leq 1-\epsilon$ for some $\epsilon>0$. With high probability as $n\to \infty$, the Betti table of the Stanley--Reisner ideal of $\Delta$ has $2$-torsion.  
\end{proposition}
\begin{proof}
The proof is the same as the proof of Theorem~\ref{thm:m torsion}, but utilizing Proposition~\ref{prop:high-prob 2tor} in place of Proposition~\ref{prop:high-probability}. 
\end{proof}

As a generalization of Question~\ref{q:2torsion threshold}, it would be interesting to understand a precise threshold on the attaching probability $p$ such that the Betti table of the Stanley--Reisner ideal of $\Delta$ does not depend on the characteristic.  A related question is posed in Question~\ref{q:threshold}.

\begin{remark}\label{rmk:2nd row}
Our constructions are based entirely on torsion in the $H_1$-groups, and thus we obtain Betti tables where the entries in the second row of the Betti table (the row of entries of the form $\beta_{i,i+2}$) depend on the characteristic.  Since Newman's work also produces small simplicial complexes where the $H_i$-groups have torsion for any $i\geq 1$~\cite[Theorem~1]{newman}, one could likely apply the methods of \S\ref{sec:construction} to produce thresholds for where the other rows of the Betti table would depend on the characteristic, and it might be interesting to explore the resulting thresholds.
\end{remark}

\section{Further Questions}\label{sec:veronese}

In the this final section, we discuss some further questions about torsion for flag complexes and for the asymptotic syzygies of geometric examples.




\begin{question}
Can one find new examples of Veronese embeddings of $\PP^r$, or of any other reasonably simple variety (Grassmanian, toric variety, etc.), whose Betti tables depend on the characteristic?  For a given $\ell$, can one produce a specific example of a variety whose Betti table has $\ell$-torsion?
\end{question}
\noindent We find it especially surprising that there are no known examples of $2$-torsion for $d$-uple embeddings of $\PP^r$.
Focusing on the case of projective space, the following question is open:
\begin{question}
What is the minimal value of $r$ such that the Betti table of the $d$-uple embedding of $\PP^r$ depends on the characteristic for some $d$?  (It is known that $2\leq r\leq 6$.)
\end{question}
An analogous question, in the context of random monomial ideals, would be as follows:

\begin{question}\label{q:threshold}
Let $m\geq 2$.  For a random flag complex $\Delta\sim \Delta(n,p)$, what is the threshold on $p$ such that the Betti table of the Stanley--Reisner ideal of $\Delta$ has $m$-torsion with high probability as $n\to \infty$?
\end{question}
A closely related result is~\cite[Theorem~8.1]{costa-farber-horak}, which implies that for any given odd prime $\ell$, the Betti table of the Stanley--Reisner ideal of $\Delta$ (with high probability as $n\to \infty$) has no $\ell$-torsion when $p \ll n^{-1/3-\epsilon}$ where $\epsilon >0$ is fixed.  
\begin{remark}\label{rmk:r bound}
We know of two natural ways that one could improve the threshold for $p$ in Theorem~\ref{thm:m torsion}.  First, one could perform a more detailed study of the essential density $m(H_m)$, as that value is surely lower than our chosen bound $\frac{1}{2}\maxdeg(X_m)$.  Second, one could aim to produce flag complexes $X_m'$ with torsion homology (not necessarily in $H_1$) whose underlying graphs have a lower essential density than $H_m$.  Of course, following the heuristic discussed in the introduction, any such improvement of the threshold for $p$ in Theorem~\ref{thm:m torsion} would suggest a corresponding improvement of the bound on $r$ in Conjectures~\ref{conj:dependence} and \ref{conj:bad primes}.
\end{remark}
In a different direction, one might ask about how large $n$ needs to be before we expect to see that the Betti table associated to $\Delta$ has $\ell$-torsion.  

\begin{question}
Fix a prime $\ell$ and $\epsilon>0$.  Let $\Delta\sim\Delta(n,p)$ be a random flag complex with $n^{-1/6}\ll p \ll 1-\epsilon$.  For a constant $0<\delta<1$, approximately how large does $n$ need to be to guarantee that
\[
\P\left[\text{ Betti table associated to $\Delta$ has $\ell$-torsion }\right] \geq 1 - \delta?
\]
\end{question}
It would be interesting to even answer this question for $2$-torsion, where the thresholds from~\cite[Theorems~7.1 and 7.2]{costa-farber-horak} make the question seemingly quite tractable.
An analogous question for Veronese embeddings of projective space would be the following:

\begin{question}
Fix a prime $\ell$ and integer $r \geq 2$. Can one provide lower/upper bounds on the minimal value of $d$ such that the Betti table of the $d$-uple embedding of $\PP^r$ has $\ell$-torsion?
\end{question}
Of course, one could ask similar questions, replacing $\PP^r$ by other varieties.  We could also turn to even more quantitative questions related to Conjecture~\ref{conj:bad primes} as well.
\begin{question}
Fix a prime $\ell$ and an integer $r \geq 2$. Can one describe the set of $d\in \ZZ$ such that the Betti table of the $d$-uple embedding of $\PP^r$ has $\ell$-torsion? Can one bound or estimate the density of this set?
\end{question}

\bibliography{bib}{}
\bibliographystyle{plain}

\end{document}